\documentclass[11pt]{article}

\usepackage{amsmath,amsthm}

\usepackage{amssymb,latexsym}

\usepackage{enumerate}

\newtheorem{tw}{Theorem}[subsection]
\newtheorem{lm}[tw]{Lemma}
\newtheorem{wn}[tw]{Corollary}
\newtheorem{stw}[tw]{Proposition}
\newenvironment{dow}{\it Proof.\rm}{\hfill $\Box$}

\theoremstyle{definition}
\newtheorem*{df}{Definition}
\newtheorem{uw}[tw]{Remark}
\newtheorem{prz}[tw]{Example}

\newcommand{\BN}{{\mathbb N}}
\newcommand{\BR}{{\mathbb R}}

\newcommand{\WW}{{\mathcal W}}
\newcommand{\FF}{{\mathcal{F}}}

\newcommand{\HH}{{\mathcal{H}}}
\newcommand{\BB}{{\mathcal{B}}}
\newcommand{\LL}{{\mathcal{L}}}

\newcommand{\MM}{{\mathcal{M}}}

\newcommand{\PP}{{\mathcal{P}}}
\newcommand{\RR}{{\mathcal{R}}}
\newcommand{\EE}{{\mathcal{E}}}
\newcommand{\bX}{{\mathbf{X}}}

\newcommand{\BRD}{{\mathbb{R}^{d}}}

\newcommand{\nsubsection}{\setcounter{equation}{0}\subsection}

\begin{document}
\title{Semi-Dirichlet forms, Feynman-Kac functionals and the
Cauchy problem for semilinear parabolic equations}
\author {Tomasz Klimsiak \smallskip\\
{\small Faculty of Mathematics and Computer Science, Nicolaus
Copernicus University} \\ {\small  Chopina 12/18,
87--100 Toru\'n, Poland}\\
{\small e-mail: tomas@mat.uni.torun.pl}}
\date{}
\maketitle
\begin{abstract}
In the first part of the paper we prove various results on
regularity of  Feynman-Kac functionals of Hunt processes
associated with time dependent semi-Dirichlet forms. In the second
part we study the Cauchy problem for semilinear parabolic
equations with measure data involving operators associated with
time-dependent forms. Model examples are non-symmetric divergence
form operators and fractional laplacians with possibly variable
exponents. We first introduce a definition of a solution
resembling Stampacchia's definition in the sense of duality and
then, using the results of the first part, we prove the existence,
uniqueness and regularity of solutions of the problem under mild
assumptions on the data.
\end{abstract}


\footnotetext{{\em Key words or phrases:} Semi-Dirichlet form,
Feynman-Kac functional, Semilinear parabolic equation, measure
data.}

\footnotetext{Tel.: +48-566112954,
E-mail: tomas@mat.umk.pl}

\nsubsection{Introduction}

Let $E$ be a locally compact separable metric space, $m$ be an
everywhere dense Borel measure on $E$  and let
$\{B^{(t)};t\in\BR\}$ be a family of regular semi-Dirichlet forms
on $L^2(E;m)$ with common domain $F$.
Let us consider a time-dependent semi-Dirichlet form
\[
\mathcal{E}(u,v)=\left\{
\begin{array}{l}(-\frac{\partial u}{\partial t},v)+\BB(u,v),
\quad(u,v)\in\WW\times L^2(0,T;F),\smallskip \\
(u,\frac{\partial v}{\partial t})+\BB(u,v),\quad(u,v)\in
L^2(0,T;F)\times\WW,
\end{array}
\right.
\]
where $\WW=\{u\in L^2(0,T;F);\,\frac{\partial u}{\partial t}\in
L^2(0,T;F')\}$, $(\cdot,\cdot)$ stands for the duality pairing
between $L^2(0,T;F)$ and $L^2(0,T;F')$, and
\[
\BB(u,v)=\int_\BR B^{(t)}(u(t),v(t))\,dt.
\]
Let $\mathbb{M}=(\{\bX_t,t\ge0\},\,\{P_z,z\in E\times\BR\})$ be a
Hunt process with life-time $\zeta$ properly associated with
$\EE$. The main object of the present paper is to study regularity
of the Feynman-Kac functionals of the form
\begin{equation}
\label{eq1.01}
u(z)=E_z\mathbf{1}_{\{\zeta>T-\tau(0)\}}\varphi(\bX_{T-\tau(0)})
+E_z\int_0^{\zeta_\tau}dA_r^\mu,\quad z\in
E_{0,T}\equiv(0,T]\times E.
\end{equation}
Here $E_z$ denotes the expectation with respect to $P_z$,
$\zeta_\tau=\zeta\wedge(T-\tau(0))$, where $\tau$ is the uniform
motion to the right, $\varphi:E\rightarrow\BR$ and $A^{\mu}$ is
the additive functional of $\mathbb{M}$ in Revuz correspondence
with a smooth measure $\mu$ on  $E_{0,T}$.

Our interest in functionals of the form (\ref{eq1.01}) comes from
the fact that  regularity of $u$ implies  regularity of solutions
of the Cauchy problem
\begin{equation}
\label{eqi.2} -\frac{\partial u}{\partial t}-L_tu=\mu,\quad
u(T)=\varphi,
\end{equation}
where $L_t$ is the operator associated with the form $B^{(t)}$.
The study of equations of the form (\ref{eqi.2}) and more general
semilinear equations of the form
\begin{equation}
\label{eqi.3} -\frac{\partial u}{\partial t}
-L_tu=f(t,x,u)+\mu,\quad u(T)=\varphi
\end{equation}
is the second main goal of the paper. We are interested in
equations with $\varphi\in L^1(E;m)$ and ``true'' measure data.
Therefore in the paper we assume that $\mu$ belongs to the space
$\mathcal{R}(E_{0,T})$  of all smooth (with respect to the
capacity determined by $\EE$) measures on $E_{0,T}$ such that
$E_zA_{\zeta_\tau}^{|\mu|}<\infty$ for $\EE$-quasi-every (q.e.)
$z\in E_{0,T}$, and that $\delta_{\{T\}}\otimes\varphi\cdot
m\in\mathcal{R}(E_{0,T})$. These are minimal assumptions on
$\mu,\varphi$ under which $u$ is finite $m_1$-a.e., and hence
finite $\EE$-q.e. The class $\mathcal{R}(E_{0,T})$ is quite wide.
If $\EE$ satisfies some duality condition (see condition
$(\Delta$) below) then it includes the space $\MM_{0,b}(E_{0,T})$
of all bounded smooth measures on $E_{0,T}$. Our general framework
of time-dependent semi-Dirichlet forms associated with the family
of semi-Dirchlet forms allows us to study (\ref{eqi.2}),
(\ref{eqi.3}) for wide class of local and nonlocal operators
$L_t$. Model examples are diffusion operators with drift terms and
fractional laplacians with constant and variable exponents (for
more examples see \cite{Fukushima,KR:JFA,KR:SM,MR,Oshima}). We
think that applicability of our general results to parabolic
equations with measure data involving nonlocal operators is of
particular interest, because to our knowledge, with the exception
of \cite{KR:JEE}, no such result has appeared in the literature.

Large majority of known  results on the regularity of $u$ given by
(\ref{eq1.01}) concerns the case where $\mu=g\cdot m_1$. One can
roughly divide them into two groups. In the first group of results
one shows that $u$ is continuous and then  that it is a viscosity
solution of (\ref{eqi.2}). To show this one assumes that
$\varphi,g$ are continuous with polynomial growth and $L_t$ is a
non-divergent form diffusion operator or L\`evy type operator with
diffusion part in the non-divergent form with Lipschitz continuous
coefficients (see, e.g., \cite{BBP,Pardoux}). The results of the
second group say that $u$ is a Sobolev space weak (in the
variational sense) solution of (\ref{eqi.2}). In the known results
$f,\varphi$ are assumed to be square integrable and $L_t$ is a
diffusion operator with regular coefficients (see \cite{BM,BL}),
uniformly elliptic diffusion operator with measurable coefficients
(see \cite{BPS,Lejay,Ro:PTRF}) or L\`evy type operator whose
diffusion part has regular coefficients (see \cite{Situ}). In
\cite{ZR} diffusion operators with singular coefficients are
considered. However, in the case considered in \cite{ZR} the
regularity of $u$ follows from that for diffusion with no singular
part and from the stochastic representation of the divergence (see
\cite{Kl:JTP,Ro:Stochastics,Stoica}).

In \cite{KR:JEE}  regularity of $u$ given by (\ref{eq1.01}) and
connections of (\ref{eq1.01}) with  solutions of (\ref{eqi.2}) are
investigated in case $L_t$ is a uniformly elliptic divergence form
operator, $\varphi\in L^1(E;m)$ and  $\mu$ is a general bounded
smooth measure. We generalize considerably  these results. The
remarkable feature of \cite{KR:JEE} and the present paper is that
in both papers the regularity of Feynman-Kac functionals of the
form (\ref{eq1.01}) plays an important role in the proof of their
connections with PDEs. Secondly, as in \cite{KR:JEE}, in the
present paper the proof of regularity of Feynman-Kac functionals
(and hence of solutions to related PDEs) relies purely on the
theory of Dirichlet forms. In the existing literature the proofs
of regularity of Feynman-Kac functionals are usually based on
results on stochastic flows (in the case of regular coefficients)
or on regularity results from the theory of PDEs combined with
approximation methods based on regularization of the data involved
in the functional.

In the first part of the paper we prove that in  general, if
$\varphi\in L^1(E;m)$ and $\mu\in\RR(E_{0,T})$, then $u$ given by
(\ref{eq1.01}) is quasi-l.s.c and quasi-c\`adl\`ag, and if $u\in
L^2(E_{0,T};m_1)$, where $m_1=dt\otimes m$, then $(0,T]\ni
t\mapsto u(t)\in L^2(E;m)$ is c\`adl\`ag. If $A^\mu$ is continuous
then $u$ is quasi-continuous, and if moreover $u\in
L^2(E_{0,T};m_1)$, then $(0,T]\ni t\mapsto u(t)\in L^2(E;m)$ is
continuous. We also show that if the following duality condition
is satisfied:
\begin{enumerate}
\item[$(\Delta)$]  for some
$\alpha\ge 0$ there exists a nest $\{F_n\}$ on $E_{0,T}$ such that for  every
$n\ge 1$ there is a non-negative $\eta_n \in L^2(E_{0,T};m_1)$
such that $\eta_n>0$ $m_1$-a.e. on $F_n$ and
$\hat{G}_\alpha^{0,T}\eta_n$ is bounded,
\end{enumerate}
where $\hat{G}_\alpha^{0,T}$ is the adjoint operator to the
resolvent $G_\alpha^{0,T}$ of the operator
$-\frac{\partial}{\partial t}-L_t$, then
\begin{equation}
\label{eq1.05} \MM_{0,b}(E_{0,T})\subset\mathcal{R}(E_{0,T}).
\end{equation}
Condition $(\Delta)$ is satisfied for instance if $\alpha\hat
G^{0,T}_{\gamma+\alpha}$ is Markovian for some $\gamma\ge0$. From
(\ref{eq1.05}) it follows in particular that if $\varphi\in
L^1(E;m)$ then $\delta_{\{T\}}\otimes\varphi\cdot
m\in\mathcal{R}(E_{0,T})$. We next prove some energy estimates for
$u$. To this end, we first prove that if $\varphi\in L^2(E;m)$ and
$\mu\in S_0(E_{0,T})$, i.e. $\mu$ is a finite energy measure on
$E_{0,T}$, then $u\in L^2(0,T;F)$ and $u$ is a weak solution of
(\ref{eqi.2}) in the variational sense. We then use this result to
show that if $\varphi\in L^1(E;m)$, $\mu\in\MM_{0,b}(E_{0,T})$ and
for some $\gamma\ge 0$ the form
$\EE_{\gamma}=\EE+\gamma(\cdot,\cdot)_L^2$ has the dual Markov
property then $u\in L^1(E_{0,T};m_1)$, $T_k(u)=((-k)\vee u)\wedge
k \in L^2(0,T;F)$ for every $k\ge 0$ and
\[
\int_0^TB^{(t)}_\gamma(T_k(u)(t),T_k(u)(t))\,dt\le k(\|\mu\|_{TV}
+\|\varphi\|_{L^1}+\gamma\|u\|_{L^1}).
\]

In the second part of the paper we study the Cauchy problems
(\ref{eqi.2}), (\ref{eqi.3}). Before describing briefly our main
results let us mention that one delicate issue one encounters when
considering (\ref{eqi.2}), (\ref{eqi.3}) with measure data is to
give proper definition of a solution. This is caused by the fact
that even in the linear case the distributional solution may be
not unique (see \cite{Serrin} for a suitable example of linear
equation with uniformly elliptic divergence form operator). The
problem of existence and uniqueness of solutions of equations with
measure data was first addressed in Stampacchia's paper
\cite{Stampacchia} devoted to the Dirichlet problem for  elliptic
equations with uniformly elliptic divergence form operator. To
overcame the difficulty with the uniqueness of solutions
Stampacchia introduced the so-called solutions  by the method of
duality and showed that in his class of solutions the problem is
well posed. A drawback to the original Stampacchia's definition of
solutions, and perhaps the main reason why the theory of solutions
by duality have not been developed, is that it applies mainly to
linear equations. In the early nineties of the last century the
so-called entropy and renormalized solutions were introduced (see,
e.g., \cite{BBGGPV,DMOP} and the references therein), and an
extensive study of nonlinear equations with measure data and local
operators began. For a selection of important results on the
subject we refer the reader to \cite{BBGGPV,DMOP} (elliptic
equations) and \cite{DPP,PPP} (parabolic equations).

In the present paper by a solution to (\ref{eqi.2}) we mean $u$
satisfying (\ref{eq1.01}). In case ($\Delta$) is satisfied we show
that equivalently  $u$ can be defined as a measurable function on
$E_{0,T}$ satisfying the equation
\begin{equation}
\label{eq1.04} (u,\eta)_{L^2(E_{0,T};m_1)}
=(\varphi,(\hat{G}^{0,T}\eta)(T))_{L^2(E;m)}
+\int_{E_{0,T}}\hat{G}^{0,T}\eta\,d\mu
\end{equation}
for every non-negative $\eta \in L^2(E_{0,T};m_1)$ such that
$\hat{G}_0^{0,T}\eta$ is bounded. It follows in particular that
under ($\Delta$) there is at most one $u$ satisfying
(\ref{eq1.04}). The definition of a solution to (\ref{eq1.01}) via
(\ref{eq1.04}) resembles Stampacchia's definition given in
\cite{Stampacchia}. In case of local operators, it  coincides with
the original definition from \cite{Stampacchia}. Note also that
our definition (\ref{eq1.04}) extends to the parabolic case and
semi-Dirichlet forms the definition introduced in \cite{KR:JFA}
(see also \cite{KR:SM}) in case  of elliptic equations with
measure data involving operators associated with Dirichlet forms.

In the semilinear case the definitions of solutions are similar to
those in the linear case. We call a measurable
$u:E_{0,T}\rightarrow\BR$ a solution to (\ref{eqi.3}) if
(\ref{eq1.01}) is satisfied with $\mu$ replaced by $f_u\cdot
m+\mu$, where $f_u=f(\cdot,\cdot,u)$. In case $(\Delta)$ is
satisfied, $u$ is a solution of (\ref{eqi.3}) if $f_u\in
L^1(E_{0,T};m_1)$  and (\ref{eq1.04}) is satisfied with $\mu$
replaced by $f_u\cdot m+\mu$.  We prove the existence and
uniqueness of solutions to  (\ref{eqi.3}) for $f$ satisfying the
monotonicity condition, continuous with respect to $u$ and such
that $f(\cdot,\cdot,0)\in \mathcal{R}(E_{0,T})$ and
\begin{equation}
\label{eqi.4} \forall_{y\in\BR}\quad f(t,x,y)\in
qL^1(E_{0,T};m_1),
\end{equation}
where $qL^1(E_{0,T};m_1)$ is the space of quasi-integrable
functions on $E_{0,T}$ (see Section 3). Let us note that equations
of the form (\ref{eqi.3}) with local operators (nonlinear of
Leray-Lions type) were considered in \cite{BBGGPV,BG}. In these
papers it is assumed that $f$ satisfies stronger than
(\ref{eqi.4}) growth condition
\begin{equation}
\label{eqi.5} \forall_{r\ge 0}\quad E_{0,T}\ni(t,x)\mapsto
\sup_{|y|\le r}|f(t,x,y)|\in L^1(E_{0,T};m_1).
\end{equation}
Elliptic problems with Laplace operator and right-hand side
satisfying weak growth condition of the form  (\ref{eqi.4}) were
considered in \cite{OP} for $f$ independent of $x$ and in
\cite{Kl:AMPA} for diagonal systems. Let us also mention the
papers \cite{BS,Blasio} in which $L$ (independent of $t$) is
assumed to be accretive on $L^1(E_{0,T};m_1)$, $\mu\in
L^1(E_{0,T};m_1)$ and $f$ satisfies some condition which implies
(\ref{eqi.5}). It is worth noting that except for \cite{Kl:AMPA}
in  all the mentioned papers $f(\cdot,\cdot,0), \mu$ are assumed
to be in $L^1(E_{0,T},m_1)$ or in $\MM_{0,b}(E_{0,T})$. In the
present paper we  consider the class $\mathcal{R}(E_{0,T})$, which
for some classes of operators defined on bounded smooth domains
$D\subset\BR^d$ includes weighted Lebesgue spaces
$L^1(D_{0,T};\delta^{\alpha}\cdot m_1)$ for some $\alpha\ge 0$,
where $\delta(x)=\mbox{dist}(x,\partial D )$. These classes of
spaces are important in applications to elliptic systems (see
\cite{QS}).

Finally, let us note that in the paper we assume that
$\{B^{(t)}\}$ appearing in the definition of $\EE$ is a family of
regular semi-Dirichlet forms. However, at the end of Section
\ref{sec4} we show that in the case where $\{B^{(t)}\}$ is a
family of non-negative quasi-regular Dirichlet forms one can apply
the so-called transfer method to the form $\EE$. Therefore all the
results of the paper on regularity of (\ref{eq1.01}) and solutions
of (\ref{eqi.2}), (\ref{eqi.3}) also hold true under the last
assumption on $\{B^{(t)}\}$.

\nsubsection{Preliminaries}
\label{sec2}

In the paper $E$ denotes a locally compact separable metric space
and  $m$ denotes an everywhere dense measure on the Borel
$\sigma$-algebra $\BB(E)$.

Let  $F$ be a dense subspace of $H\equiv L^{2}(E; m)$ and $B:
F\times F\rightarrow\BR$ be a bilinear form. We say that $B$ is
closed  on $F$ if
\begin{enumerate}
\item[(B1)]there exists $\alpha_{0}\ge 0$ such that
\[
B_{\alpha_{0}}(u, u)\ge 0, \quad u\in F,
\]
where $B_{\alpha_{0}}(u, v)=B(u, v)+\alpha_{0}(u,v)_{L^{2}}$,
\item[(B2)]there exists $K\ge 0$ such that
\[
|B(u,v)|\le K B_{\alpha_{0}}(u, u)^{1/2}B_{\alpha_{0}}(v,v)^{1/2},
\quad u,v\in F,
\]
\item[(B3)] $F$ is a Hilbert space with the inner product
\[
(u,v)_{F}\equiv \frac12 (B_{\alpha_{0}}(u,v) +B_{\alpha_{0}}(v,u)).
\]
\end{enumerate}
We say that $B$  has the Markov property if
\begin{enumerate}
\item[(B4)]for all $u\in F$ and  $a\ge 0$, $u\wedge a\in F$ and
$B(u\wedge a, u-u\wedge a)\ge 0$.
\end{enumerate}
We say that $B$ has the dual Markov property if
\begin{enumerate}
\item[($\hat{\mbox{B}}$4)]for all $u\in F$ and $a\ge 0$, $u\wedge a\in F$ and
$ B(u-u\wedge a,u\wedge a)\ge 0$.
\end{enumerate}

We say that a form $(B,F)$ is a Dirichlet form if it is closed and
has  the Markov property (B4). A Dirichlet form $(B,F)$ is called
non-negative if $\alpha_{0}=0$.

It is known (see \cite[Theorem 1.1.5]{Oshima}) that if $(B,F)$ is
a Dirichlet form then (B4) is equivalent to the following
condition: $\alpha G_{\alpha}^{0}$ is Markovian for every
$\alpha>0$, i.e. if $0\le f\le 1$ then $0\le\alpha
G_{\alpha}^{0}f\le 1$, where $\{\alpha G_{\alpha}^{0},
\alpha>\alpha_{0}\}$ is the resolvent associated with $(B,F)$.

We say that $(B,F)$ is regular if there exists a subset
$\mathcal{C}$ of the space $C_0(E)$ of continuous functions on $E$
with compact support such that $F\cap\mathcal{C}$ is
$B_{\alpha_{0}}$-dense in $F$ and  dense in $C_{0}(E)$ with
uniform norm.

For $k\ge 0$ put
\[
T_{k}(u)=\max\{\min\{u, k\}, -k\}, \quad u\in\BR.
\]
Then for every $\alpha>\alpha_{0}$ and $u\in F$,
\[
\alpha(T_{k}(u)- \alpha G_{\alpha}^{0} T_{k}(u),
u-T_{k}(u))_{L^{2}}\ge 0,
\]
because $-k\le \alpha G_{\alpha}^{0} T_{k}(u)\le k$ by the
Markovovian property of $\alpha G_{\alpha}^{0}$. By Theorems 1.1.4
and  1.1.5 in \cite{Oshima} the above inequality implies that
\begin{enumerate}
\item[(B4a)]$T_{k}(u)\in F$ for every $u\in F$ and
\[
B(T_{k}(u), T_{k}(u))\le B(T_{k}(u), u).
\]
\end{enumerate}

It follows that if $(B, F)$ is closed then condition (B4) is
equivalent to (B4a). Similarly, if $(B, F)$ is closed then
$(\hat{\mbox{B}}4)$ is equivalent to
\begin{enumerate}
\item[($\hat{\mbox{B}}$4a)]$T_{k}(u)\in F$ for every $u\in F$ and
\[
B(T_{k}(u), T_{k}(u))\le B(u, T_{k}(u)).
\]
\end{enumerate}

In what follows $E^{1}=\BR\times E$, $m_{1}=\lambda^{1}\otimes m$
and $\lambda^{1}$ is the Lebesgue measure on $\BR$. We set
$\FF=L^{2}(\BR;F)$, $\FF_{0,T}=L^{2}(0,T;F)$,
$\FF_{T}=L^{2}(-\infty,T;F)$ and $\HH=L^{2}(E^{1};m_{1})$,
$\HH_{0, T}=L^{2}(0,T;H)$, $\HH_{T}=L^{2}(-\infty,T;H)$. Let
$\FF'=L^2(\BR;F')$ denotes the dual space to $\FF$. We set
\[
\WW=\{u\in\FF; \frac{\partial u}{\partial t}\in\FF'\}
\]
and define $\WW_{T}, \WW_{0,T}$ analogously to $\WW$ but with
$\FF$ replaced  by $\FF_{T}, \FF_{0, T}$, respectively. For
$u\in\WW$ we put
\[
\|u\|_{\WW}=\|\frac{\partial u}{\partial t}\|_{\FF'} +
\|u\|_{\FF}.
\]
The norms  $\|u\|_{\WW_{0, T}}$ and $\|u\|_{\WW_{T}}$) are defined
analogously: we replace $\FF$ in the above definition by $\FF_{0,
T}$ and $\FF_{T}$, respectively.

For $a,b\in\BR\cup\{+\infty\}\cup\{-\infty\}$ let  $C(a, b; H)$
denote  the space of all functions $u\in\BB((a,b]\times E)$ such
that the mapping $(a,b]\ni t\mapsto u(t)\in H$ is continuous and
let $C(\BR; H)= C(-\infty,+\infty;H)$. It is well known (see
\cite{LM}) that $\WW\subset C(\BR; H)$.

By  $D(a,b;H)$ we denote the space of those functions
$u\in\BB((a,b]\times E)$  for which the mapping $(a,b]\ni t\mapsto
u(t)\in H$ is c\`adl\`ag, i.e. right continuous with left limits.

Let $\{B^{(t)},t\in\BR\}$ be a family of regular Dirichlet forms
on $F$. In the paper we assume that for every $u, v\in F$ the
mapping
\[
\BR\ni t\mapsto B^{(t)}(u, v)
\]
is measurable and the constant $\alpha_{0}$ of conditions (B2),
(B3) does not depend on $t$. We may and will assume that
$\alpha_{0}<1$. We also assume that there exists $\lambda>0$ such
that
\begin{equation}
\label{eqp.0} \frac1\lambda B_{\alpha_{0}}^{(0)}(u, u)\le
B_{\alpha_{0}}^{(t)}(u, u)\le \lambda B_{\alpha_{0}}^{(0)}(u, u),
\quad u\in F,\quad t\in\BR.
\end{equation}
For $(u,v)\in(\FF\times\WW)\cup(\WW\times\FF)$ we put
\begin{equation}
\label{eq2.02} \EE(u, v)= \left\{
\begin{array}{l}(-\frac{\partial u}{\partial t}, v)
+\BB(u, v), \quad (u, v)\in\WW\times\FF, \smallskip \\
(u, \frac{\partial v}{\partial t}) +\BB(u, v), \quad (u, v)\in\FF\times\WW
\end{array}
\right.
\end{equation}
and $\EE_{\alpha}(u, v)=\EE(u,v)+\alpha(u, v)_{L^2}$, where
$(\cdot,\cdot)$ stands for the duality pairing between $\FF$ and
$\FF'$ and
\[
\BB(u,v)=\int_{\BR}B^{(t)}(u(t), v(t))\,dt.
\]

It is known (see, e.g., \cite{Oshima,Stannat} that for every
$\alpha>\alpha_{0}$ and $f\in\HH$ there exist unique $G_{\alpha}f,
\hat{G}_{\alpha}\in\FF$ such that
\[
\EE_{\alpha}(G_{\alpha}f, v)=(f, v), \quad v\in\WW,
\]
\[
\EE_{\alpha}(v, \hat{G}_{\alpha}f)=(f, v), \quad v\in\WW.
\]
Moreover, there exist strongly continuous semigroups $\{T_{t},
t\ge 0\}$, $\{\hat{T}_{t}, t\ge 0\}$ on $\HH$ such that
$\|T_{t}\|_{L^{2}}\le e^{\alpha_{0}t}$,
$\|\hat{T}_{t}\|_{L^{2}}\le e^{\alpha_{0}t}$, $t\ge 0$, and
\begin{equation}
\label{eqp.1}
G_{\alpha}f=\int_{0}^{\infty}e^{-\alpha t}T_{t}f\,dt, \quad
\hat{G}_{\alpha}f=\int_{0}^{\infty}e^{-\alpha t}\hat{T}_{t}f\,dt.
\end{equation}
It is also known that $T_{t}$ (resp. $\hat{T}_{t}$) can be
extended  to  $L^{\infty}(E^{1};m_{1})$ (resp. $L^{1}(E^{1};
m_{1})$) and that $T_{t}$ (resp. $\hat{T}_{t}$) is a contraction
on $L^{\infty}(E^{1};m_{1})$ (resp. $L^{1}(E^{1};m_{1})$).
Therefore $\alpha G_{\alpha}$ (resp. $\alpha\hat{G}_{\alpha}$)
given by (\ref{eqp.1}) is well defined and is a contraction on
$L^{\infty}(E^{1};m_{1})$ (resp. $L^{1}(E^{1};m_{1})$).

We say that a form $\EE$ has the dual Markov property if ($\hat{\mbox{B}}4)$
holds with $B$ replaced by $\EE$ and $F$ replaced by $\mathcal{W}$.
This is equivalent to say that $\alpha\hat{G}_\alpha$ is a contraction on $L^1(E^1;m_1)$
for every $\alpha>0$ (see the reasoning in the proof of \cite[Theorem 1.1.5]{Oshima}).
By standard approximation arguments (see the proof of Theorem \ref{tw3.8}), if $\EE$
has the dual Markov property then
\[\int_0^T B^{(t)}(u(t)-(u\wedge a)(t), (u\wedge a)(t))\ge 0\]
for every $u\in \FF_{0,T}$ and $a\in\mathbb{R}$.

A function $u\in\BB^{+}(E^{1})$ satisfying $\beta
G_{\beta+\alpha}u\le u$  (resp. $\beta \hat{G}_{\beta+\alpha}u\le
u$) for every $\beta\ge 0$ is called an $\alpha$-excessive (resp.
$\alpha$-co-excessive) function. By $\PP_{\alpha}$ (resp.
$\hat{\PP}_{\alpha}$) we denote the set of all $\alpha$-excessive
(resp. $\alpha$-co-excessive) functions.

Let $\psi\in L^{2}(E^{1}; m_{1})$, $0<\psi\le 1$, $m_{1}$-a.e. For
an open set $U\subset E^{1}$ we put
\[
\mbox{Cap}_{\psi}(U)\equiv(h_{U}, \psi)_{L_2},
\]
where $h=G_{1}\psi$ and $h_{U}$ is the reduced function of $h$ on
$U$ (see \cite{Stannat}). For an arbitrary set $B\subset E^{1}$ we
put
\[
\mbox{Cap}_{\psi}(B)=\inf\{\mbox{Cap}_{\psi}(U); B\subset U,
U\subset E^{1}, U\mbox{-open}\}.
\]

We say that set $B$ is $\EE$-exceptional if Cap$_{\psi}(B)=0$. We
say that some property is satisfied quasi everywhere (q.e.) if the
set of those $z\in E^{1}$ for which it does not hold is
$\EE$-exceptional. The capacity Cap$_{\psi}$ is equivalent to the
capacity considered in \cite[Section 6.2]{Oshima}. It is known
(see the argument following (6.2.2) on page 237 in \cite{Oshima})
that for every $f\in\WW$,
\begin{equation}
\label{eq2.4} \|e_{f}\|_{\FF}\le c\|f\|_{\WW},
\end{equation}
where
\[
e_{f}=\min\{u\in\PP_{1}\cap\FF: u\ge f\mbox{ a.e.}\}
\]
(see \cite[Theorem 6.2.6]{Oshima}). By \cite[Proposition
3.6]{Stannat}  (see also the reasoning in the proof of
\cite[Proposition 3.7]{Stannat}), for every $u\in\HH$ and
$\lambda>0$,
\[
\mbox{Cap}_{\psi}(\{|u|>\lambda\})\le \frac 1 \lambda
\|u\|_{L^{2}}\cdot\|\psi\|_{L^{2}}.
\]
Combining the above with (\ref{eq2.4}) we conclude that for every
$f\in\WW$,
\begin{equation}
\label{eq1.1c} \mbox{Cap}_{\psi}(\{|f|>\lambda\})\le \frac c
\lambda \|f\|_{\WW}\cdot\|\psi\|_{L^{2}}.
\end{equation}

Let $\{F_{k}\}$ be an increasing sequence of closed subsets of
$E^{1}$.  It is called a nest if Cap$_{\psi}(F_{k}^{c})\rightarrow
0$ as $k\rightarrow\infty$. We say that a function $u\in
\BB(E^{1})$ is quasi-continuous (resp. quasi-l.s.c.) if there
exists a nest $\{F_{k}\}$ such that $u_{|F_{k}}$ is continuous
(resp. l.s.c.) for every $k\ge 1$.

A Borel measure $\mu$ on $E^{1}$ is called smooth if it does not
charge $\EE$-exceptional sets and there exists a nest $\{F_{k}\}$
such that $|\mu|(F_{k})<\infty$ for $k\ge 1$, where $|\mu|$ is the
variation of $\mu$. By $S$ we denote the set of all smooth
measures on $E^1$. $S_{0}$ is the set of all measures of finite
energy integrals, i.e. the subset of $S$ consisting of all
measures $\mu$ having the property that there is $K\ge0$ such that
\[
\int_{E}|\eta|\,d|\mu|\le K\|\eta\|_{\WW}, \quad \eta\in\WW.
\]

Let us remark that each $\eta\in\WW$ possesses a quasi-continuous
version, so that the integral on the left-hand side of the above
inequality is well defined.

For a given Borel measure $\mu$ on $E^{1}$ and a Borel measurable
function $f$ on $E^{1}$ let $f\cdot\mu$ denote the Borel measure
on $E^{1}$ given by the formula
\[
(f\cdot\mu)(B)=\int_{B}f\,d\mu, \quad B\in \BB(E^{1}).
\]
We will also use the notation
\[
\langle f, \mu\rangle=\int_{E^{1}}f\,d\mu.
\]
Since Cap$_\psi$ is strongly subadditive (see \cite{Stannat}),
using and (\ref{eq1.1c}) and repeating the proofs of Lemmas 2.2.8,
2.2.9 in \cite{Fukushima} (it is enough to replace the capacity
appearing there by Cap$_\psi$) one can show that for every $\mu\in
S$ there exists a nest $\{F_{k}\}$ such that
$\mathbf{1}_{F_{k}}\cdot\mu\in S_{0}$.

Let us  recall that for every $\mu\in S_{0}$ and
$\alpha>\alpha_{0}$  there exists unique $U_{\alpha}\mu,
\hat{U}_{\alpha}\mu\in\FF$ such that
\[
\EE_{\alpha}(U_{\alpha}\mu, \eta)
=\EE_{\alpha}(\eta,\hat{U}_{\alpha}\mu)=\int_{E^{1}}\eta\,d\mu
\]
for every $\eta\in\WW$.

It is known (see \cite{MR,Oshima}) that with a regular Dirichlet
form $(B,F)$ one can associate a Hunt process
$\mathbb{M}^0=(\{X_{t},t\ge 0\}, \{P_{x},x\in
E\cup\{\Delta\}\},\FF^{0},\{\theta_{t}^{0}, t\ge 0\},\zeta^0)$
such that for every $f\in\BB_{b}(E)$ and $\alpha>0$ the function
\[
(R_{\alpha}^{0}f)(x)=E_{x}\int_{0}^{\infty}e^{-\alpha
t}f(X_{t})\,dt,\quad x\in E
\]
is a $B$-quasi-continuous $m$-version of $G^{0}_{\alpha}$. It is
also known (see \cite{Oshima,Stannat}) that with the
time-dependent form $\EE$ defined by (\ref{eq2.02}) one can
associated a Hunt process $\mathbb{M}=(\{\bX_{t},t\ge0\},\{P_{z},
z\in E^{1}\cup\{\Delta\}),\FF,\{\theta_{t},t\ge 0\},\zeta)$ such
that for every $f\in\BB_{b}(E^{1})$ and $\alpha>0$ the function
\[
(R_{\alpha}f)(z)=E_{z}\int_{0}^{\infty}e^{-\alpha
t}f(\bX_{t})\,dt,\quad z\in E^1
\]
is an $\EE$-quasi-continuous $m_{1}$-version of $G_{\alpha}f$.
Moreover,
\[
\bX_{t}=(\tau(t), X_{\tau(t)}), \quad t\ge 0,
\]
where $\tau(t)$ is the uniform motion to the right, i.e.
$\tau(t)=\tau(0)+t$, $\tau(0)=s$, $P_{z}$-a.s. for $z=(s,x)$, and
for each $s\in\BR$ the process $\mathbb{M}^{(s)}=(\{X_{t+s}, t\ge
0\}, \{P_{s,x}, x\in E\}, \FF^{(s)}=\{\FF_{s+t}, t\ge 0\})$ is a
Hunt process associated with the form $(B^{(s)}, F)$.

A real valued $\FF$ adapted process $A$ is called an additive
functional (AF) of $\mathbb{M}$ if there exists a set
$\Lambda\subset\Omega$ (called defining set) and an
$\EE$-exceptional set $N\subset E^{1}$ such that
$P_{z}(\Lambda)=1$, $z\in E_{1}\setminus N$,
$\theta_{t}\Lambda\subset\Lambda$, $t\ge 0$ and for every
$\omega\in\Lambda$,
\begin{enumerate}
\item[(a)]$[0,\infty)\ni t\mapsto A_{t}(\omega)$ is c\`adl\`ag,
\item[(b)]$A_{0}(\omega)=0$, $|A_{t}(\omega)|<\infty$, $t\in [0,\zeta)$,
$A_{t}(\omega)=A_{\zeta}(\omega)$, $t\ge\zeta(\omega)$,
\item[(c)]$A_{s+t}(\omega)=A_{t}(w)+A_{s}(\theta_{t}\omega)$, $s,t\ge 0$.
\end{enumerate}
An AF $A$ of $\mathbb{M}$ is called natural (NAF) if $A$ and $\bX$
has no common discontinuities. An AF $A$ of $\mathbb{M}$ is called
continuous (CAF) if $A$ is a continuous process. Finally, an AF
$A$ of $\mathbb{M}$ is called positive (PAF) if $A_t$ is
non-negative for every $t\ge 0$.

Let $\mu$ be a non-negative Borel measure on $E^{1}$ and $A$  be a
PAF of $\mathbb{M}$. We say that $\mu$ and $A$ are in the Revuz
correspondence if for every $m_{1}$-integrable
$\alpha$-co-excessive function $h$ and every $f\in
\BB_{b}^{+}(E_{1})$,
\[
\int_{E^{1}}f(z)h(z)\mu(dz) = \lim_{t\rightarrow 0}\frac{1}{t}
E_{h\cdot m_1}(f\cdot A)_{t}=
\lim_{\beta\rightarrow\infty}\beta(h, U^{\beta}_{A}f)_{L^{2}},
\]
where
\[
(f\cdot A)_{t}=\int_{0}^{t}f(\bX_{s})\,dA_{s},
\quad (U_{A}^{\beta}f)(z)
=E_{z}\int_{0}^{\infty}e^{-\beta t}f(\bX_{t})\,dA_{t}.
\]

By \cite[Theorem 6.4.7]{Oshima}, for every $\mu\in S_{0}$ there
exists  a unique NAF $A$ of $\mathbb{M}$ in the Revuz
corespondence with $\mu$. Since each measure $\mu\in S$ may be
approximated by measures in $S_{0}$, repeating step by step the
proof of \cite[Theorem 4.1.16]{Oshima} one can show that for every
$\mu\in S$ there exists a unique NAF $A$ of $\mathbb{M}$ in the
Revuz correspondence with $\mu$. We will denote it by $A^{\mu}$.

\nsubsection{Feynman-Kac functionals}
\label{sec3}

In this section we prove basic regularity results for $u$ defined
by (\ref{eq1.01}). We begin with continuity properties. Then we
prove that $u$ is the usual weak solution to (\ref{eqi.2}) if
$\varphi\in L^2(E;m)$ and $\mu\in S_0(E_{0,T})$. In the last part
we derive energy estimates for $u$ in  case $\varphi\in L^1(E;m)$,
$\mu\in\MM_{0,b}(E_{0,T})$.

\subsubsection{General continuity properties}

Let $E_T=(-\infty,T]\times E$, $E_{0,T}=(0,T]\times E$.  By
$\tilde{C}(E^{1})$ (resp. $\tilde{C}(E_{T}), \tilde{C}(E_{0, T})$)
we denote the set of Borel measurable functions $u$ on $E^{1}$
(resp. $E_{T},E_{0,T}$) such that for $m_{1}$-a.e. $z\in E^{1}$
(resp. $E_{T}, E_{0,T}$) the process $t\mapsto u(\bX_{t})$ is
right continuous on $[0,\zeta)$ (resp. $[0,\zeta_{\tau})$, where
$\zeta_{\tau}=\zeta\wedge(T-\tau(0))$) and the process $t\mapsto
u(\bX_{t-})$ is left continuous on $(0, \zeta)$ (resp. $(0,
\zeta_{\tau})$) $P_{z}$-a.s.

In \cite{MR} it is proved that if $B$ is a quasi-regular Dirichlet
form and $u\in \tilde{C}(E)$ then $u$ is quasi-continuous. From
this it follows that for every finely open $U\subset E$, if $u\in
\tilde{C}(U)$ then $u$ is quasi-continuous on $U$ (to see this it
is enough to consider the part of the form $B$ on $U$, which is
also quasi-regular). If $U$ is not finely open then in general the
last implication does not hold. In the following lemma we show
that it is true, however, if $U=E_{T}$.

\begin{lm}
\label{lm1.1}
Assume that $f\in\tilde{C}(E_{T})$. Then $f$ is quasi-continuous on $E_{T}$.
\end{lm}
\begin{dow}
We may assume that $f\ge 0$. Let us extend $f$ by zero to the
whole $E^{1}$. As in the proof of \cite[Lemma V.2.6]{MR}, with
\cite[Proposition V.1.6]{MR} replaced by \cite[Proposition
IV.3.4]{Stannat}, we show that for every open $B\subset E^1$,
\begin{equation}
\label{eq1.1} \mbox{Cap}_{\psi}(B)=E_{\mu}e^{-S_{B}},
\end{equation}
where
\[
S_{B}=\inf\{t\ge 0,  \bar{X}_{0}^{t}\cap B\neq\emptyset\}.
\]
We next repeat, with some obvious changes, arguments from the
proof of \cite[Lemma V.2.19]{MR} to show that (\ref{eq1.1}) holds
for every $B\in\mathcal{B}(E^{1})$. Since $\mathbb{M}$ is special
standard, $S_{B}=\inf\{0\le t<\zeta;X_{t}\in A\mbox{ or }X_{t-}\in
A\}\wedge\zeta$.  For $f\in\BB_{b}(E^{1})$ set
\[
\|f\|=E_{\mu}\sup_{t\ge 0} e^{-t}(|f(\mathbf{X}_{t})| \vee
|f(\bX_{t-})|),
\]
\[
\|f\|_{T}=E_{\mu,T}\sup_{t\ge 0}
e^{-t}(|f(\mathbf{X}_{t})|\vee|f(\bX_{t-})|),
\]
where $P_{\mu, T}(\cdot)=\int_{E_{T}}P_{z}(\cdot)\varphi (z)\,dz$,
$P_{\mu}(\cdot)=\int_{E}P_{z}(\cdot)\varphi(z)\,dz$. Arguing as in
the proof of  \cite[Lemma 5.23]{MR} (with the norm $\|\cdot\|_{T}$
on $\tilde{C}(E_{T})$) we show that
$\overline{C_{b}(E_{T})}=\tilde{C}(E_{T})$, where
$\overline{C_{b}(E_{T})}$ is the closure of $C_{b}(E_{T})$ in
$\tilde{C}(E_{T})$ with respect to the norm $\|\cdot\|_{T}$. Let
$\{f_{n}\}\subset C_{b}(E_{T})$ be such that
$\|f-f_{n}\|_{T}\rightarrow 0$ and $\|f_{n+1}-f_{n}\|_{T} <
2^{-2n}$, $n\ge 1$. For $n,N\in\BN$ set
\[
A_{n}=\{z\in E_{T}; \, |f_{n+1}-f_{n}|>2^{-n}\}, \quad
B_{N}=\bigcup_{n\ge N}A_{n}.
\]
By (\ref{eq1.1}), for every $N\in\mathbb{N}$,
\begin{align*}
\mbox{Cap}_{\psi}(B_{N})\le \sum_{n\ge N}\mbox{Cap}_{\psi}(A_{n})
= \sum_{n\ge N}E_{\mu}e^{-S_{B}}&=\sum_{n\ge N}
\|\mathbf{1}_{A_{n}}\|= \sum_{n\ge N}\|\mathbf{1}_{A_{n}}\|_{T}\\
&\le \sum_{n\ge N}2^n\|f_{n+1}-f_{n}\|_{T}\le 2^{-N+1}.
\end{align*}
By standard argument (see the proof of \cite[Proposition
5.24]{MR}) we can now show that the function
\[
\tilde{f}(z)=\left\{
\begin{array}{ll}\lim_{n\rightarrow\infty}f_{n}(z), &
z\in\bigcup_{N\in\mathbb{N}}(E_{T} \setminus B_{N}),\smallskip \\ f(z), &
\mbox{otherwise}
\end{array}
\right.
\]
is quasi-continuous on $E_{T}$ and $\tilde{f}=f$ q.e. on $E_{T}$.
\end{dow}
\medskip

Let $S(E_T), S(E_{0,T})$ denote the spaces of smooth measures with
support in $E_T, E_{0,T}$, respectively. By $\mathcal{R}$ (resp.
$\mathcal{R}(E_{T}),\mathcal{R}(E_{0,T})$) we denote the space of
all $\mu\in S$ (resp. $\mu\in S(E_{T})$, $\mu\in S(E_{0,T})$) such
that $E_{z}A^{|\mu|}_{\infty}<\infty$ for q.e. $z\in E^{1}$ (resp.
$E_{T},E_{0,T}$). Observe that if $\mu\in S(E_{T})$ then
$\mathcal{R}=\mathcal{R}(E_{T})$.

We say that a Borel measurable function $u$ is quasi-c\`adl\`ag on
$E^{1}$ (resp. $E_{T},E_{0,T}$) if for q.e. $z\in E^{1}$ (resp.
 $E_{T}$, $E_{0,T}$) the process $t\mapsto u(\bX_{t})$ is
c\`adl\`ag on $[0,\zeta]$ (resp. $[0,\zeta_{\tau}]$) $P_{z}$-a.s.

\begin{stw}
\label{stw1.1} Assume that $\delta_{\{T\}}\otimes\varphi\cdot m$,
$\mu\in\mathcal{R}(E_{0, T})$ and $\varphi\ge 0$, $\mu\ge 0$.
Let $u:E_T\rightarrow\BR$ be defined as
\begin{equation}
\label{eq1.0}
u(z)=E_{z}\mathbf{1}_{\{\zeta>T-\tau(0)\}}\varphi(\bX_{T-\tau(0)})
+E_{z}\int_{0}^{\zeta_{\tau}}dA_{r}^{\mu}, \quad z\in E_{T}.
\end{equation}
Then
\begin{enumerate}
\item[\rm{(i)}] $u$ is quasi-l.s.c. and quasi-c\`adl\`ag, and if $u\in
L^{2}(E_{T};m_{1})$  then $u\in D(-\infty,T;H)$. If moreover
$A^{\mu}$ is continuous on $[0,\zeta_{\tau}]$ then $u$ is
quasi-continuous on $E_{T}$, and if $u\in L^{2}(E_{T};m_{1})$ then
$u\in C(-\infty,T H)$,

\item[\rm{(ii)}] there exists a MAF $M$ such that
\begin{equation}
\label{eq1.5}
u(\bX_{t})=\mathbf{1}_{\{\zeta>T-\tau\{0\}\}}\varphi(\bX_{T-\tau(0)})
+\int_{t}^{\zeta_{\tau}}dA_{r}^{\mu}-
\int_{t}^{\zeta_{\tau}}dM_{r}, \quad t\in[0,\zeta_{\tau}]
\end{equation}
$P_z$-a.s. for q.e. $z\in E^1$.
\end{enumerate}
\end{stw}

\begin{dow}
Let us consider the following additive functional
\begin{equation}
\label{eq1.2}
A_{t}=\mathbf{1}_{\{\tau(0)<T\le \tau(0)+t\}}\varphi(\bX_{T-\tau(0)}),
\quad t\ge 0.
\end{equation}
Observe that $A=A^{\nu}$, where
$\nu=\delta_{\{T\}}\otimes\varphi\cdot m$.  Set $\delta=\nu+\mu$
and
\begin{equation}
\label{eq1.3} w(z)=E_{z}A^{\delta}_{\infty}, \quad z\in E^{1}.
\end{equation}
By the assumptions, $w(z)<\infty$ for a.e. $z\in E^{1}$. Using
argument analogous  to that in the proof of \cite[Lemma
4.2]{KR:JFA} one can show that in fact  $w(z)<\infty$ for q.e.
$z\in E^{1}$. Observe that
\begin{equation}
\label{eq1.4} w(z)=u(z), \quad z\in(-\infty,T)\times E.
\end{equation}
Since for every $B\in\mathcal{B}(E)$, Cap$_{\psi}(\{T\}\times
B)=0$  iff $m(B)=0$, it follows from (\ref{eq1.4}) that
$u(z)<\infty$ for q.e. $z\in E_{T}$. Let $N=\{z\in E^{1};
u(z)=\infty\}$. We may assume that $N$ is properly exceptional. By
the strong Markov property, for every $z\in E_{T}\setminus N$ and
$\sigma\in\mathcal{T}$ such that $0\le\sigma\le\zeta_{\tau}$ we
have
\[
u(\bX_{\sigma})=E_{z}(\mathbf{1}_{\{\zeta>T-\tau\{0\}\}}
\varphi(\bX_{T-\tau(0)})|\FF_{\sigma})
+E_{z}(A_{\infty}^{\mu}|\FF_{\sigma})-A_{\sigma}.
\]
By the section theorem it follows that $u(\bX)$ has the
representation (\ref{eq1.5}) with $M$ being a c\`adl\`ag version
of the martingale
\[
E_{z}(\mathbf{1}_{\{\zeta>T-\tau\{0\}\}}\varphi(\bX_{T-\tau(0)})
+\int_{0}^{\zeta_{\tau}}dA_{r}^{\mu}|\FF_{t}) -u(\bX_{0}).
\]
This shows (ii) because by \cite[Lemma A.3.5]{Fukushima} one can
choose such a version independently of $z$. From (\ref{eq1.5}) we
conclude that $u$ is quasi-c\`adl\`ag on $E_{T}$. Now let us
assume additionally that $A^{\mu}$ is continuous on
$[0,\zeta_{\tau}]$. Then from (\ref{eq1.5}) we deduce that
$u(\bX)$ is right-continuous on $[0, \zeta_{\tau}]$,
$(u(\bX))_{-}$ is left-continuous on $[0, \zeta_{\tau})$ and
$(u(\bX_{t}))_{-}=u(\bX_{t-})$, $t\in[0,\zeta_{\tau})$ (see the
reasoning in the proof of Claim 2 in \cite[Proposition
IV.5.14]{MR}). Hence $u\in\tilde{C}(E_{T})$, which when combined
with Lemma \ref{lm1.1} implies that $u$ is quasi-continuous on
$E_{T}$. For $\alpha>0$ put
\[
u_{\alpha}(z)=E_{z}\mathbf{1}_{\{\zeta>T-\tau(0)\}}
\varphi(\bX_{T-\tau(0)})+\alpha E_{z}
\int_{0}^{\zeta_{\tau}}e^{-\alpha t}u(\bX_{t})\,dt.
\]
By what has already been proved, $u_{\alpha}$ is quasi-continuous
on $E_{T}$. By the strong Markov property,
\[
u_{\alpha}(z)=E_{z}\mathbf{1}_{\{\zeta>T-\tau(0)\}}
\varphi(\bX_{T-\tau(0)})+E_{z}\int_{0}^{\zeta_{\tau}}
(1-e^{-\alpha t})\,dA_{t}^{\mu},
\]
which implies that $u_{\alpha}(z)\nearrow u(z)$, $z\in E_{T}$.
Hence $u$ is quasi-l.s.c on $E_{T}$. If $\alpha>\alpha_{0}$ and
$u\in L^{2}(E^{1}; m)$ then $w\in L^{2}(E^{1}, m_{1})$ (since
$w\le u$) and $w_{\alpha}$ defined as
\[
w_{\alpha}(z)=\alpha E_{z}\int_{0}^{\infty}e^{-\alpha t}
w(\bX_{t})\,dt, \quad z\in E^{1}
\]
belongs to $\mathcal{W}(E^{1})$. Since $w_{\alpha}$ is
quasi-continuous, $w_{\alpha}\in C(\mathbb{R};H)$. Moreover, since
$w_{\alpha}=\alpha R_{\alpha}w$,
\[
\EE_{\alpha}(w_{\alpha}, \eta)=\alpha(w, \eta), \quad \eta\in
\mathcal{W}(E^{1})
\]
or, equivalently,
\begin{equation}
\label{eq1.6} \EE(w_{\alpha}, \eta)=\alpha(w-w_{\alpha}, \eta),
\quad \eta\in \mathcal{W}(E^{1}).
\end{equation}
Let $t_{1}<t_{2}$ and let $v$ be a measurable function on $E^{1}$.
Put
\begin{equation}
\label{eq1.6a}
v^{\varepsilon}(t, x)=\left\{
\begin{array}{ll}\frac 1 \varepsilon v(t_{1}+\varepsilon, x)(t-t_{1}),
& (t,x)\in [t_{1}, t_{1}+\varepsilon]\times E, \smallskip \\ v(t,x),
& (t,x)\in[t_{1}, t_{2}]\times E, \smallskip \\ -\frac 1 \varepsilon
v(t_{2}, x)(t-t_{2}-\varepsilon), & (t,x)\in[t_{2},
t_{2}+\varepsilon]\times E, \smallskip \\ 0, &\mbox{otherwise}.
\end{array}
\right.
\end{equation}
Then taking $w_{\alpha}^{\varepsilon}$ as a test function  in
(\ref{eq1.6})  and letting $\varepsilon\rightarrow 0^{+}$ we get
\begin{align}
\label{eq1.7} \nonumber &\|w_{\alpha}(t_{1})\|^{2}_{L^{2}}-
\|w_{\alpha}(t_{2})\|^{2}_{L^{2}}+
2\int_{t_{1}}^{t_{2}}B^{(t)}(w_{\alpha}(t), w_{\alpha}(t))\,dt \\
&\qquad =2\alpha\int_{t_{1}}^{t_{2}}\|w(t)-w_{\alpha}(t)\|^{2}_{L^{2}}\,dt,
\end{align}
whereas taking $\eta^{\varepsilon}$ with nonnegative $\eta\in F$
as a test function and letting $\varepsilon\rightarrow 0^{+}$ we
get
\begin{align}
\label{eq1.8} \nonumber &(w_{\alpha}(t_{1}), \eta)_{L^{2}}-
(w_{\alpha}(t_{2}), \eta)_{L^{2}}+
\int_{t_{1}}^{t_{2}}B^{(t)}(w_{\alpha}(t), \eta)\,dt\\
&\qquad=\alpha\int_{t_{1}}^{t_{2}}(w(t)-w_{\alpha}(t),
\eta)_{L^{2}}\,dt.
\end{align}
Write
\[
x(t)=2\int_{0}^{t}B^{(s)}(w(s), w(s))\,ds, \quad x^{\eta}(t)
=\int_{0}^{t}B^{(s)}(w(s), \eta)\,ds,
\]
\[
x_{\alpha}(t)=2\int_{0}^{t}B^{(s)}(w_{\alpha}(s), w_{\alpha}(s))\,ds,
\quad x^{\eta}_{\alpha}(t)=\int_{0}^{t}B^{(s)}(w_{\alpha}(s),\eta)\,ds,
\]
and
\[
y_{\alpha}(t)=\|w_{\alpha}(t)\|^{2}_{L^{2}}-x_{\alpha}(t), \quad y(t)
=\|w(t)\|^{2}_{L^{2}}-x(t),
\]
\[
y^{\eta}_{\alpha}(t)=(w_{\alpha}(t),
\eta)_{L^{2}}-x_{\alpha}^{\eta}(t), \quad y^{\eta}(t)=(w(t),
\eta)_{L^{2}}-x^{\eta}(t).
\]
By what has already been proved, $w_{\alpha}(z)\nearrow w(z)$,
$z\in E^{1}$. It is also known that $w_{\alpha}\rightarrow w$ in
$\FF$ (see \cite[Theorem 6.1.2]{Oshima}). Therefore from
(\ref{eq1.7}), (\ref{eq1.8}) it may be concluded that
\[
y_{\alpha}(t)\rightarrow y(t), \quad y^{\eta}_{\alpha}(t)
\rightarrow y^{\eta}(t), \quad t\in\mathbb{R},
\]
\[
x_{\alpha}(t)\rightarrow x(t), \quad x^{\eta}_{\alpha}(t)\rightarrow x^{\eta}(t),
\quad t\in\mathbb{R}.
\]
Moreover, $y$ is nonincreasing and for every $\eta\in F$ such that
$\eta\ge 0$ the function $y_{\eta}$ is nonincreasing. Since the
sequences $\{\|w_{\alpha}(t)\|^{2}_{L^{2}}\}$, $\{(w_{\alpha}(t),
\eta)_{L^{2}}\}$ are nondecreasing we get by \cite{Peng} that the
mappings $t\mapsto \|w(t)\|^{2}_{L^{2}}$, $t\mapsto (w(t),
\eta)_{L^{2}}$ are c\`adl\`ag on $\mathbb{R}$. By the classical
results they are also l.s.c. We now show that $w\in D(\mathbb{R},
H)$. Let $t_{n}\rightarrow t_{0}^{+}$. Then
\[
\|w(t_{n})-w(t_{0})\|^{2}_{L^{2}}=\|w(t_{n})\|^{2}+\|w(t_{0})\|^{2}-
2(w(t_{n}), w(t_{0}))_{L^{2}}.
\]
Since $t\rightarrow\|w(t)\|^{2}_{L^{2}}$ is c\`adl\`ag,
\[
\limsup_{n\rightarrow\infty} \|w(t_{n})-w(t_{0})\|^{2}_{L^{2}}=
2\|w(t_{0})\|^{2}- 2\liminf_{n\rightarrow\infty}(w(t_{n}), w(t_{0}))_{L^{2}}.
\]
But the mapping $t\rightarrow(w(t),w(s))$ is l.s.c. Hence
\[
\limsup_{n\rightarrow\infty}\|w(t_{1})-w(t_{0})\|^{2}_{L^{2}}\le
0.
\]
Let $t_{n}\nearrow t_{0}^{-}$. Since
$t\rightarrow\|w(t)\|^{2}_{L^{2}}$  is locally bounded and
$t\rightarrow(w(t), \eta)_{L^{2}}$ is c\`adl\`ag, it follows that
there exists $v\in H$ not depending on the choice of the sequence
$\{t_{n}\}$ such that $w(t_{n})\rightarrow v$ weakly in $H$. By
\cite{Pierre} there exists an $m_{1}$-version $\tilde{w}$ of $w$
such that the mapping $\BR\ni t\mapsto\tilde{w}(t)\in H$ is
c\`agl\`ad, i.e. left continuous with right limits. Without loss
of generality we may assume that $\tilde{w}(t_{n})=w(t_{n})$
$m$-a.e. for $n\ge 1$. Therefore $\{w(t_{n})\}$ is strongly
convergent in $H$ and of course $w(t_{n})\rightarrow v$ in $H$. In
particular, $\|w(t_{n})\|_{L^{2}}\rightarrow \|v\|_{L^{2}}$. Since
$t\mapsto\|w(t)\|^{2}_{L^{2}}$ is c\`adl\`ag, there exists the
limit $\lim_{t\rightarrow t_{0}^{-}}\|w(t)\|^{2}_{L^{2}}$ and
obviously $\lim_{t\rightarrow t_{0}^{-}}\|w(t)\|_{L^{2}}=\|v\|$.
Therefore $\lim_{t\rightarrow t_{0}^{-}}w(t)=v$ strongly in $H$.
Finally, since $w(z)=u(z)$ for $z\in (-\infty,T)\times E$,  $u\in
D(-\infty, T; H)$.
\end{dow}

\begin{uw}
If in Lemma \ref{lm1.1} and Proposition \ref{stw1.1} we consider
the form $\EE$ on $[0,\infty)\times E$ instead the form $\EE$ on
$E^{1}$, then their assertions remains valid if we replace $E_{T}$
by $E_{0, T}$, replace $(-\infty, T]$ by $(0,T]$ and $\mathcal{R}$
by $\mathcal{R}(E_{0,T})$.
\end{uw}

\subsubsection{Energy estimates: the case of finite energy integral
measures}

In the sequel  $\int_{a}^{b}$ stands for $\int_{(a,b]}$.

\begin{df}
Let $\varphi\in L^{2}(E;m)$ and $\mu\in S_{0}(E_{0,T})$. We say
that a measurable function $u: E_{0,T} \rightarrow\mathbb{R}$ is a weak
solution of the Cauchy problem
\begin{equation}
\label{eq1.10}
-\frac{\partial u}{\partial t} -L_{t}u=\mu, \quad u(T)=\varphi
\end{equation}
if
\begin{enumerate}
\item[(a)] $u\in\FF_{0,T}$, $u\in D(0,T;H)$,
\item[(b)] for every  $t\in(0, T]$ and
$\eta\in\mathcal{W}(E_{0,T})$,
\begin{align*}
&(u(t),\eta(t))_{L^{2}}+ \int_{t}^{T}(u(s), \frac{\partial
\eta}{\partial t}(s))_{L^{2}}\,ds+ \int_{t}^{T}B^{(s)}(u(s),
\eta(s))\,ds \\
& \qquad=(\varphi, \eta(T))_{L^{2}}
+\int_{t}^{T}\!\!\int_{E}\eta(z)\,d\mu(z).
\end{align*}
\end{enumerate}
\end{df}

\begin{stw}
\label{stw1.2} There exists at most one weak solution of
$\mbox{\rm(\ref{eq1.10})}$.
\end{stw}
\begin{dow}
Without loss of generality we may assume that $\alpha_{0}=0$.
Assume  that $u_{1}, u_{2}$ are solutions of (\ref{eq1.10}) and
set $u=u_1-u_2$. Then for every $\eta\in\mathcal{W}(E_{0, T})$ and
$t\in [0,T]$,
\begin{equation}
\label{eq1.11} (u(t), \eta(t))_{L^{2}}+ \int_{t}^{T}(u(s),
\frac{\partial \eta}{\partial s}(s))_{L^{2}}\,ds
+\int_{t}^{T}B^{(s)}(u(s),\eta(s))\,ds=0.
\end{equation}
From this we easily deduce that $u\in\mathcal{W}(E_{0,T})$.
Replacing $\eta$ by $u$ in (\ref{eq1.11}) we get
\[
\|u(t)\|^{2}_{L^{2}}+ 2\int_{t}^{T}B^{(s)}(u(s), u(s))\,ds=0,
\]
which implies that $u=0$ a.e.
\end{dow}

\begin{tw}
\label{tw1.1} Assume that $\varphi\in L^{2}(E;m)$ and $\mu\in
S_{0}(E_{0, T})$. Then $u:E_{0,T}\rightarrow\BR$ defined by
\mbox{\rm(\ref{eq1.0})} is a weak solution of the Cauchy problem
\mbox{\rm(\ref{eq1.10})}.
\end{tw}
\begin{dow}
Let $\nu=\delta_{\{T\}}\otimes\varphi\cdot m$ and
$\eta\in\mathcal{W}$.  Then
\begin{align*}
\nu(\eta)=\int_{E^{1}}\eta(z)\,
\nu(dz)&= \int_{E}\eta(T,x)\varphi(x)\,dx \\
& \le\|\eta(T)\|_{L^{2}}\cdot \|\varphi\|_{L^{2}}
\le \sup_{t\ge 0}\|\eta(t)\|_{L^{2}}\|\varphi\|_{L^{2}}\le
\|\eta\|_{\mathcal{W}}\cdot \|\varphi\|_{L^{2}}.
\end{align*}
Hence $\nu\in S_{0}$. Let $\alpha>\alpha_{0}$ and let
$\mu^{\alpha}=e^{-\alpha(T-\cdot)}\cdot\mu$,
$\delta^{\alpha}=\nu+\mu^{\alpha}$. Observe that
\[
A_{t}^{\mu_{\alpha}}=\int_{0}^{t}e^{-\alpha(T-\tau(r))}\,dA_{r}^{\mu},
\quad t\ge 0.
\]
Put
\[
w_{\alpha}(z)
=E_{z}\int_{0}^{\infty}e^{-\alpha t}\,dA_{t}^{\delta_{\alpha}}.
\]
It is known that $w_{\alpha}=U_{\alpha}\delta^{\alpha}$, i.e.
$\EE_{\alpha}(w_{\alpha},\eta)=\langle\delta^{\alpha},\eta\rangle$,
$\eta\in\mathcal{W}$. Hence
\[
(w_{\alpha}, \frac{\partial\eta}{\partial t})_{L^{2}}
+\mathcal{B}(w_{\alpha}, \eta)= \langle\delta^{\alpha},
\eta\rangle-\alpha(w_{\alpha}, \eta)_{L_2}, \quad
\eta\in\mathcal{W}.
\]
Therefore for any $t\in[0,T]$ and  $\varepsilon>0$ we have
\[
(w_{\alpha}, \frac{\partial\eta^{\varepsilon}}{\partial t})
+\mathcal{B}(w_{\alpha}, \eta^{\varepsilon})=
\langle\delta^{\alpha}, \eta^{\varepsilon}\rangle
-\alpha(w_{\alpha}, \eta^{\varepsilon}), \quad \eta\in\mathcal{W},
\]
where the approximation $\eta^{\varepsilon}$ is defined by
(\ref{eq1.6a}) with $t_1=t, t_2=T$. Letting
$\varepsilon\rightarrow 0^{+}$ and using Proposition \ref{stw1.1}
we get
\begin{align*}
&(w_{\alpha}(t), \eta(t))_{L^{2}}+ \int_{t}^{T}(w_{\alpha}(s),
\frac{\partial\eta}{\partial s}(s))_{L^{2}} +
\int_{t}^{T}B^{(s)}(w_{\alpha}(s), \eta(s))\,ds \\
&\qquad=\int_{t}^{T}\eta(z)\delta^{\alpha}(dz)-
\int_{t}^{T}\alpha(w_{\alpha}(s), \eta(s))_{L^{2}}\,ds
\end{align*}
for every $\eta\in\mathcal{W}(E_{0,T})$. The second term on the
left-hand side of the above equation is equal to
\[
-\alpha\int_{t}^{T}(u_{\alpha}(s), e^{\alpha(T-s)}\eta(s))\,ds
+\int_{t}^{T}(u_{\alpha}(s),e^{\alpha(T-s)}\frac{\partial\eta}{\partial s})\,ds.
\]
Putting $\tilde{u}(t)=e^{\alpha(T-t)}u_{\alpha}(t)$ we conclude that
\begin{align*}
&(\tilde{u}(s), \eta(t))_{L^{2}}+ \int_{t}^{T}(\tilde{u}(s),
\frac{\partial\eta}{\partial s}(s))_{L^{2}}\,ds
+\int_{t}^{T}B^{(s)}(\tilde{u}(s), \eta(s))\,ds \\
&\qquad =(\varphi, \eta(T))_{L^{2}}+\int_{t}^{T}\eta(z)\,d\mu(z)
\end{align*}
for every $\eta\in\mathcal{W}(E_{0, T})$. Let us put $w(t)
=e^{\alpha(T-t)}w_{\alpha}(t)$, $t\in[0, T]$. Then for $z=(s,x)$,
$s\in[0,T]$ we have
\begin{align*}
w(z)=e^{\alpha(T-s)}
E_{z}\int_{0}^{\infty}e^{-\alpha t}\,dA_{t}^{\delta_{\alpha}}
&=e^{\alpha(T-s)}E_{z}\int_{0}^{\infty}e^{-\alpha
t}e^{-\alpha(T-\tau(t))}\,dA_{t}^{\delta}\\
& =e^{\alpha(T-s)}E_{z}\int_{0}^{\infty}
e^{-\alpha t}e^{-\alpha(T-t-s)}\,dA_{t}^{\delta}
=E_{z}A_{\infty}^{\delta}.
\end{align*}
Since $u(z)=w(z)$ for $z\in[0, T)\times E$ and $u(T)=\varphi$,
$u(z)=\tilde{u}(z)$, $z\in E_{0,T}$.
\end{dow}

\subsubsection{Energy estimates: the case of bounded smooth measures}

Let $\WW_T=\{u\in\WW(E_{0,T});\,u(T)=0\}$,
$\WW_0=\{u\in\WW(E_{0,T});\,u(0)=0\}$ and let
\[
\EE^{0,T}(u,v)=\left\{
\begin{array}{ll}\int_0^T\left(-\frac{\partial u}{\partial t},v\right)\,dt+
 \int_0^TB^{(t)}(u(t),v(t))\,dt,\quad (u,v)\in\WW_T\times\FF_{0,T},\medskip \\
 \int_0^T\left(u,\frac{\partial v}{\partial t}\right)\,dt+
 \int_0^TB^{(t)}(u(t),v(t))\,dt,\quad (u,v)\in\FF_{0,T}\times\WW_0.
 \end{array}
\right.
\]
It is known (see \cite[Example I.4.9(iii)]{Stannat}) that
$\EE^{0,T}$ is a generalized semi-Dirichlet form and
\[
\LL=-\frac{\partial}{\partial t}-L_{t}, \quad D(\LL)=\{u\in \WW_T;
\,\LL u\in\HH_{0,T} \}
\]
is the operator associated with $\EE^{0,T}$. Note that the adjoint operator
$\hat{\LL}$ to $\LL$ is given by
\[\hat{\LL}=\frac{\partial}{\partial t}-\hat{L}_{t}, \quad
D(\hat{\LL})=\{u\in \WW_0;\,\hat{\LL} u\in\HH_{0, T}\}.
\] Let $\{T_t^{0,T},\,t>0\}$ (resp. $\{\hat{T}_t^{0,T},\,t\ge0\}$) be a
$C_0$-semigroup on $\HH_{0,T}$ associated with the operator $\LL$
(resp. $\hat{\LL}$). By (\ref{eqp.0}),
$\|T_t^{0,T}\|_{L^2\rightarrow L^2}\le e^{\alpha_0t}$,
$\|\hat{T}_t^{0,T}\|_{L^2\rightarrow L^2}\le e^{\alpha_0t}$, $t\ge
0$, and the corresponding resolvents are given by
\begin{equation}
\label{eqd.1}
G^{0, T}_{\alpha}f=\int_{0}^{\infty}e^{-\alpha t}T_t^{0,T}f\,dt, \quad
\hat{G}^{0, T}_{\alpha}f=\int_{0}^{\infty}e^{-\alpha t}\hat{T}_{t}^{0, T}f\,dt.
\end{equation}
The Hunt process $\mathbb{M}^{0,T}$ properly associated with the
form $\EE^{0,T}$ is the process $\mathbb{M}$ with life-time
$\zeta_\tau$. Therefore $T^{0,T}_t=0$, $\hat{T}^{0,T}_t=0$, $t\ge
T$. It follows that $G_\alpha^{0,T},\hat{G}_\alpha^{0,T}$ are well
defined as operators on $L^2(E_{0,T};m_1)$  for every
$\alpha\ge0$. It is standard that $T^{0,T}$ (resp.
$\hat{T}^{0,T}$) can be extended to $L^\infty(E_{0,T};m_1)\cup
L^2(E_{0,T};m_1)$ (resp. $L^1(E_{0,T};m_1)\cup L^2(E_{0,T};m_1)$)
and that the extension of $T^{0,T}_t$ (resp. $\hat{T}^{0,T}_t$) is
a contraction on $L^\infty(E_{0,T};m_1)$ (resp.
$L^1(E_{0,T};m_1)$). Therefore for every $\alpha\ge0$ we can
extend $G^{0,T}_\alpha$ (resp. $\hat{G}^{0,T}_\alpha$) defined by
(\ref{eqd.1}) to an operator on $L^\infty(E_{0,T};m_1)$ (resp.
$L^1(E_{0,T};m_1)$). For $\mu\in S$ and $\alpha\ge 0$ we define
\[
R^{0,T}_{\alpha}\mu(z)=E_{z}\int_{0}^{\zeta_\tau}e^{-\alpha
t}dA_{t}^{\mu}, \quad \hat{R}^{0,
T}_{\alpha}\mu(z)=\hat{E}_{z}\int_{0}^{\hat{\zeta_\tau}}e^{-\alpha
t}d\hat{A}_{t}^{\mu},\quad z\in E^1,
\]
where $\hat{\zeta_{\tau}}=\zeta\wedge \tau(0)$,  $\hat{A}^{\mu}$ is the dual additive functional associated
with $\mu$ (see \cite{Oshima}). It is clear that for every $f\in
L^\infty(E_{0,T};m_1) \cup L^2(E_{0,T};m_1)$ and  $g\in
L^1(E_{0,T};m_1) \cup L^2(E_{0,T};m_1)$,
\begin{equation}
\label{eq.eq111}
R_\alpha^{0,T}f=G_\alpha^{0,T}f,\quad
\hat{R}_\alpha^{0,T}g=\hat{G}_\alpha^{0,T}g
\end{equation}
$m_1$-a.e. for every $\alpha\ge0$, and for every non-negative
$\mu,\nu\in S$,
\[
\langle R_\alpha^{0, T}\mu, \nu\rangle=\langle\mu,
\hat{R}_\alpha^{0, T}\nu\rangle.
\]
for $\alpha\ge0$.

By (\ref{eq.eq111}) and Proposition \ref{stw1.1}, for every $f\in
L^\infty(E_{0,T};m_1) \cup L^2(E_{0,T};m_1)$ and  $g\in
L^1(E_{0,T};m_1)\cup L^2(E_{0,T};m_1)$ the resolvents
$G^{0,T}_{\alpha}f,\, \hat{G^{0,T}_{\alpha}}g$ have
quasi-continuous $m_1$-versions. In what follows we adopt the
convention that they are already quasi-continuous. If $\alpha=0$
then we write $R^{0,T},R^{0,T},G^{0,T},\hat{G}^{0,T}$ instead of
$R^{0,T}_0,R^{0,T}_0, G^{0,T}_0$, $\hat{G}^{0,T}_0$.

\begin{stw}
\label{stw1.5} Assume that $\EE$ satisfies the dual condition
$(\Delta)$.  Then
\begin{equation}
\label{eq3.20} \MM_{0, b}(E_{0, T})\subset \mathcal{R}(E_{0, T}).
\end{equation}
\end{stw}
\begin{dow}
Let $\{\eta_n\}$ be the sequence of the definition of condition
$(\Delta)$. Since
\begin{equation}
\label{eq1.13b} (G^{0,T}_{\alpha}\mu, \eta_{n})_{L^2}=\langle\mu,
\hat{G}^{0,T}_{\alpha}\eta_{n}\rangle \le\|\mu\|\cdot
\|\hat{G}^{0,T}_{\alpha}\eta_{n}\|_{\infty},
\end{equation}
it follows that $R^{0,T}\mu<\infty$ $m_{1}$-a.e. on $E_{0, T}$.
\end{dow}

\begin{uw}
\label{uw3.7} If for some $\gamma\ge 0$ the form $\EE_{\gamma}$
has the dual Markov property then the duality condition $(\Delta)$
is satisfied. Indeed, by \cite[Theorem 1.1.5]{Oshima}, $\alpha\hat
G^{0,T}_{\gamma+\alpha}$  is Markovian for every $\alpha>0$, so
$(\Delta)$ is satisfied with $\eta\equiv1$, $\alpha=1$ and
$F_{n}=E$, $ n\ge 1$.
\end{uw}

The following example shows that the condition that $\EE_{\gamma}$
has the dual Markov property for some $\lambda\ge0$ is not
necessary for $(\Delta)$ to hold.

\begin{prz}
\label{prz3.8} Let $D\subset\BRD$, $d\ge3$, be an open bounded set
with smooth boundary, and let
\[
\LL=-\frac{\partial}{\partial t}-L_t=-\frac{\partial}{\partial t}-
\sum^d_{i,j=1}\frac{\partial}{\partial x_{j}}
(a_{ij}\frac{\partial}{\partial x_{i}})
+\sum^d_{i=1}b_{i}\frac{\partial}{\partial x_{i}}\,,
\]
where $a_{ij}, b_i: [0,T]\times D\rightarrow\BR$ are measurable
functions such that $b_i$ is bounded, $a_{ij}=a_{ji}$ and
\[
\lambda^{-1} |\xi|^{2} \le \sum^d_{i,j=1}a_{ij}\xi_{i}\xi_{j}\le
\lambda|\xi|^{2}, \quad \xi=(\xi_1,\dots,\xi_d)\in\BRD
\]
for some $\lambda\ge1$. Then
\begin{align*}
G^{0, T}f(z)&=E_{z}\int_{0}^{\infty}f(\bX_{t})\,dt =
E_{s,x}\int_{0}^{(T-\tau(0))\wedge\zeta}f(s+t, X_{s+t})\,dt\\
&=E_{s,x}\int_{0}^{T-s}f(s+t, X^{D}_{s+t})\,dt
=\int_{0}^{T-s}\!\!\int_{D}p_{D}(s+t, x,t, y)f(s+t, y)\,dt\,dy\\
& =\int_{0}^{T}\!\!\int_{D}p_{D}(s, x,t, y)f(t, y)\,dt\,dy,
\end{align*}
where $X^{D}$ denotes the process $X$ killed upon leaving $D$  and
$p_D$ is the transition density of $X^D$.  We also have
\[
\hat{G}^{0,T}f(s,x)=\int_{0}^{s}\!\!\int_{D}p_{D}(t,y,s,x)f(t,y)\,dt\,dy.
\]
Let $f\equiv 1$. Then by Aronson's estimates,
\[
(\hat{G}^{0, T}1)(s,x)\le c_{1}
\int_{0}^{s}\!\!\int_{D}(s-t)^{-d/2}
\exp\left(\frac{-c_{2}|y-x|^{2}}{2(s-t)}\right)\,dt\,dy \le c'T.
\]
Thus condition $(\Delta)$ is satisfied. On the other hand it follows from
the formula preceding \cite[Corollary 1.5.4]{Oshima} (page 33) that if we take $b(x)=\sqrt{|x|}$
and $D=B(0,1)$ then there is no  $\gamma\in\mathbb{R}$ such that  $\EE_{\gamma}$ has the dual Markov property.
\end{prz}

\begin{wn}
\label{wn3.9} If for some $\gamma\ge 0$ the form $\EE_{\gamma}$
has the dual Markov property then \mbox{\rm(\ref{eq3.20})} is
satisfied.
\end{wn}
\begin{proof}
Follows from Proposition \ref{stw1.5} and Remark \ref{uw3.7}.
\end{proof}

\begin{tw}
\label{tw3.8} Assume that $\varphi\in L^1(E;m)$,
$\mu\in\mathcal{M}_{0,b}(E_{0,T})$ and there exists
$\gamma\ge\alpha_{0}$ such that $\EE_{\gamma}$ has the dual Markov
property. Let $u$ be defined by \mbox{\rm(\ref{eq1.0})}. Then
$u\in L^{1}(E_{0, T};m_{1})$, $T_{k}(u)\in\FF_{0,T}$ and
\begin{equation}
\label{eq3.15}
\int_{0}^{T}B^{(t)}_{\gamma}(T_{k}(u)(t),T_{k}(u)(t))\,dt \le
k(\|\mu\|+\|\varphi\|_{L^{1}}+\gamma\|u\|_{L^{1}})
\end{equation}
for every $k\ge0$, and moreover, for every $\alpha>0$,
\begin{equation}
\label{eq3.16} \|u\|_{L^1}\le
\alpha^{-1}e^{T(\alpha+\gamma)}(\|\varphi\|_{L^1}+\|\mu\|).
\end{equation}
\end{tw}
\begin{dow}
Since $\EE_{\gamma}$ has the dual Markov property,
$\alpha\hat{R}^{0,T}_{\alpha+\gamma}1\le1$ for $\alpha>0$, which
implies that $\hat{R}^{0,T}1\le\alpha^{-1}e^{T(\alpha+\gamma)}$.
Since $u=R^{0,T}\nu$ $m_{1}$-a.e. on $E_{0,T}$, where
$\nu=\delta_{T}\otimes\varphi\cdot m+\mu$, we therefore have
\[
\|u\|_{L^1}=(R^{0,T}\nu,1)_{L^2} =\langle
\nu,\hat{R}^{0,T}1\rangle \le
\alpha^{-1}e^{T(\alpha+\gamma)}(\|\varphi\|_{L^1}+\|\mu\|),
\]
which proves (\ref{eq3.16}). To prove (\ref{eq3.15}), let us
consider a nest $\{F_{n}\}$ such that
$\varphi_{n}=\mathbf{1}_{F_{n}}(T, \cdot)\varphi\in L^{2}(E;m)$,
$\mu_{n}=\mathbf{1}_{F_{n}}\cdot\mu\in S_{0}$. Write
\[
u_{n}(z)=E_{z}\mathbf{1}_{\{\zeta>T-\tau(0)\}}
\varphi_{n}(\bX_{T-\tau(0)})+
E_{z}\int_{0}^{\zeta_{\tau}}\,dA_{r}^{\mu_{n}}.
\]
By Theorem \ref{tw1.1}, $u_{n}\in\FF\cap D(0, T; H)$ and for every
$\eta\in\WW(E_{0,T})$,
\begin{align}
\label{eq1.12a} \nonumber &(u_{n}(0), \eta(0))_{L^{2}}
+\int_{0}^{T}(u_{n}(s), \frac{\partial\eta}{\partial s}(s))_{L^{2}}\,ds
+ \int_{0}^{T}B^{(s)}(u_{n}(s), \eta(s))\,ds\\ &\qquad=(\varphi_{n},
\eta(T))_{L^{2}}+\int_{E_{0, T}}
\eta\,d\mu_{n}.
\end{align}
Given $w\in\FF_{0,T}\cap D(0, T; H)$ set
\[
[w]_{m}(t)=\int_{-T}^{t}me^{-m(t-s)}w(s)\,ds.
\]
It is well known that $[w]_{m}\rightarrow w$ in $\FF_{0, T}$ and
$[w]_{m}(t)\rightarrow w(t)$ in $H$ for every $t\in[0, T]$ (in the
proof of the last property it is usually assumed that $w\in C(0,
T; H)$ but in fact it is enough to know that $w\in D(0, T; H)$).
Moreover, $[w]_{m}\in\mathcal{W}(E_{0, T})$,
\begin{equation}
\label{eq1.12c}
([w]_{m})_{t}=m(w-[w]_{m})
\end{equation}
and $[w]_{m}\le w$, $[w]_{m}\le [w]_{m+1}$, $m\ge 1$. Let us fix
$n\in\mathbb{N}$ for a moment  and put $v=u_{n}$. Taking
$\eta=[T_{k}(v)]_{m}$ as a test function in (\ref{eq1.12a}) we
obtain
\begin{align}
 \label{eq1.12b}
\nonumber &(v(0),[T_{k}(v)]_{m}(0))_{L^{2}}+ \int_{0}^{T}(v(s),
([T_{k}(v)]_{m})_{s}(s))_{L^{2}}\,ds\\ &\qquad\quad +\int_{0}^{T}B^{(s)}(v(s),
([T_{k}(v)]_{m})_{s}(s))\,ds \nonumber\\ &\qquad= (\varphi_{n},
[T_{k}(v)]_{m}(T))_{L^{2}}+ \int_{E_{0, T}}[T_{k}(v)]_{m}(z)\mu_{n}(dz)
\end{align}
and, by (\ref{eq1.12c}),
\begin{align}
\label{eq1.12f} \nonumber &\int_{0}^{T}(v(s),
([T_{k}(v)]_{m})_{s}(s))_{L^{2}}\,ds=
\int_{0}^{T}([T_{k}(v)]_{m}(s),
([T_{k}(v)]_{m})_{s}(s))_{L^{2}}\,ds \\ & \nonumber
\qquad\quad +\int_{0}^{T}(v(s)-[T_{k}(v)]_{m}(s),
([T_{k}(v)]_{m})_{s}(s))_{L^{2}}\,ds  \\ &
\nonumber \qquad=I_{1}(m)+ \int_{0}^{T}(v(s)-[T_{k}(v)]_{m}(s),
T_{k}(v)(s)-[T_{k}(v)]_{m}(s))_{L^{2}}\,ds \\ & \qquad=I_{1}(m)+I_{2}(m).
\end{align}
Let us denote by $F$ the integrand in the integral $I_{2}(m)$.
Observe that
\begin{equation}
\label{eq1.12d}
-k\le [T_{k}(v)]_{m}(z)\le k, \quad z\in E_{0, T}.
\end{equation}
If $-k\le v(z)\le k$ then  $F(z)\ge 0$. If $v(z)>k$ then by (\ref{eq1.12d}),
\[
F(z)=(v(z)-[T_{k}(v)]_{m}(z), k-[T_{k}(v)]_{m}(z))\ge 0.
\]
Similarly, $F(z)\ge 0$ if $v(z)<-k$. Therefore $I_{2}(m)\ge 0$ for
$m\in\mathbb{N}$. Moreover,
\[
I_{1}(m)=\frac 1 2 \|[T_{k}(v)]_{m}(T)\|^{2}_{L^{2}}-
\frac 1 2 \|[T_{k}(v)]_{m}(0)\|^{2}_{L^{2}}.
\]
By the above equality, (\ref{eq1.12b})--(\ref{eq1.12d}) and the
convergence properties of the sequence $\{[T_{k}(v)]_{m}\}$ we get
\[
\int_{0}^{T}B^{(s)}(v(s), T_{k}(v(s)))_{L^{2}}\,ds \le
k(\|\varphi\|_{L^{1}}+\|\mu\|).
\]
By the above inequality and the assumptions,
\[
\int_{0}^{T}B_{\gamma}^{(s)}(T_{k}(u_{n}(s)),
T_{k}(u_{n}(s)))\,ds \le
k(\|\varphi\|_{L^{1}}+\|\mu\|)+k\gamma\|u\|_{L^{1}}\,.
\]
Letting $n\rightarrow\infty$  we get (\ref{eq3.15}).
\end{dow}

\begin{stw}
\label{stw1.4} Assume that $\mu,\nu$ are non-negative smooth
measures on $E_{0,T}$ and
\[
E_{z}\int_{0}^{\zeta_{\tau}}dA_{t}^{\mu}
\le E_{z}\int_{0}^{\zeta_{\tau}}dA_{t}^{\nu}
\]
for q.e. $z\in E_{0,T}$. If $\EE$ has the dual Markov  property
then $\|\mu\|_{TV}\le \|\nu\|_{TV}$.
\end{stw}
\begin{dow}
It is well known that for every measurable $h$ on $E_{0, T}$  such
that $\eta\ge h$ $m_{1}$-a.e. for some $\eta\in\WW(E_{0, T})$
there exists (a unique) minimal solution $u\in \FF_{0, T}$ of the
obstacle problem
\begin{equation}
\label{eq1.13a}
\hat{\LL}u=0\mbox{ on }\{u>h\}, \quad \hat{\LL}u\ge 0,
\quad u(0)=0, \quad u\ge h
\end{equation}
(see \cite{MP,Pierre}). By Riesz's theorem there exists a Radon
measure  $\delta$ on $E_{0, T}$ such that $\hat{\LL}u=\delta$ (in
$C_{0}(E_{0, T})$). From this one can easily deduce that
$\delta\in S_{0}(E_{0, T})$. Therefore $u=\hat{R}^{0, T}\delta$.
Since $\EE$ is regular, there exists a sequence $\{E_{n}\}$ of
compact subsets of $E$ with the property that for each
$n\in\mathbb{N}$ there exists $\eta\in\WW(E_{0, T})$ such that
$\eta_{n}\ge h_{n}\equiv\mathbf{1}_{[0,T]\times E_{n}}$. Therefore
for every $n\in\mathbb{N}$ there exists a solution $u_{n}$ of the
obstacle problem (\ref{eq1.13a}) with the barrier $h_{n}$. Let
$\delta_{n}\in S_{0}(E_{0, T})$ be such that $u_{n}=\hat{R}^{0,
T}\delta_{n}$. Since $u_{n}$ is the smallest potential majorizing
$h_{n}$ such that $u_{n}(0)=0$, $u_{n}\le u_{n}\wedge 1$.
Therefore $u_{n}=1$ q.e. on $[0, T]\times E_{n}$, which implies
that $u_{n}\nearrow 1$ q.e. on $(0,T]\times E$. By the
assumptions, $R^{0, T}\mu\le R^{0, T}\nu$. Hence
\begin{align*}
\|\mu\|_{TV}=\lim_{n\rightarrow \infty}\langle\mu, u_{n}\rangle
&=\lim_{n\rightarrow\infty}\langle \mu, \hat{R}^{0,T}
\delta_{n}\rangle =\lim_{n\rightarrow\infty}\langle R^{0, T}\mu,
\delta_{n}\rangle\\
& \le \lim_{n\rightarrow\infty}\langle R^{0, T}\nu,
\delta_{n}\rangle =\lim_{n\rightarrow\infty}\langle \nu,
\hat{R}^{0, T}\delta_{n}\rangle=\|\nu\|_{TV},
\end{align*}
which proves the proposition.
\end{dow}

\nsubsection{Linear equations with measure data}
\label{sec4}

In this section we consider linear problems of the form
(\ref{eqi.2}) under the assumption that $\EE$ satisfies the
duality condition. The case of general forms will be considered in
a more general setting of semilinear equations in the next
section.

\begin{df}
Let $\varphi\in L^{1}(E;m)$, $\mu\in\mathcal{M}_{0, b}(E_{0,T})$
and assume that $\EE$ satisfies the duality condition $(\Delta)$.
We say that a measurable function $u:E_{0,T} \rightarrow\BR$ is a
solution of (\ref{eq1.10}) in the sense of duality if $u\in
L^{1}(E_{0,T};\eta\cdot m_1)$ and
\begin{equation}
\label{eq3.22} (u,\eta)_{L^2}=(\varphi,
\hat{G}^{0,T}\eta(T))_{L^{2}} +\int_{E_{0,T}}\hat{G}^{0,T}\eta\,d\mu
\end{equation}
for every non-negative $\eta \in L^2(E_{0,T};m_1)$ such that
$\hat{G}^{0,T}\eta$ is bounded.
\end{df}

\begin{stw}
\label{stw1.6} Let $\mu,\varphi,\EE$ be as in the above definition
and  let $u:E_{0,T}\rightarrow\BR$ be defined by
\mbox{\rm{(\ref{eq1.0})}}. Then $u$ is a unique solution of
$\mbox{\rm{(\ref{eq1.10})}}$ in the sense of duality.
\end{stw}
\begin{dow}
Uniqueness easily follows from condition $(\Delta)$. Let
$\{F_{n}\}$ be a generalized nest such that
$\mu_{n}=\mathbf{1}_{F_{n}}\cdot\mu\in S_{0}$ and
$\varphi_{n}=\mathbf{1}_{F_{n}}\varphi\in L^{2}(E;m)$. Set
\[
u_{n}(z)=E_{z}\mathbf{1}_{\{\zeta>T-\tau(0)\}}
\varphi_{n}(\bX_{T-\tau(0)})+
E_{z}\int_{0}^{\zeta_{\tau}}dA_{t}^{\mu_{n}}.
\]
By Theorem \ref{tw1.1}, $u_{n}\in\FF_{0,T}\cap D(0,T; H)$ and for
every $\psi\in\WW(E_{0,T})$,
\begin{align*}
&(u_{n}(0), \psi(0))_{L^{2}}+ \int_{0}^{T}(u_n(t), \frac{\partial
\psi}{\partial t}(t))_{L^{2}}\,dt + \int_{0}^{T}B^{(t)}(u_{n}(t),
\psi(t))\,dt\\
&\qquad= (\varphi_{n}, \psi(T)) +\int_{E_{0,T}}\psi(z)\,d\mu_{n}(z).
\end{align*}
Taking $\psi=\hat{G}^{0, T}\eta$ with non-negative $\eta\in
L^2(E_{0,T};m_1)$ such that $\hat{G}^{0,T}\eta$ is bounded as a
test function we obtain
\begin{equation}
\label{eq1.14} (u_{n}, \eta)_{L^2}= (\varphi_{n},
\hat{G}^{0, T}\eta(T))_{L^{2}}+ \int_{E_{0, T}}\hat{G}^{0,
T}\eta(z)\,d\mu_{n}(z).
\end{equation}
Observe that $|u_{n}|\le v$, $u_{n}\rightarrow u$ $m_1$-a.e, where
$v(z)=G^{0, T}\nu$ and $\nu=\delta_{T}\otimes|\varphi|\cdot
m+|\mu|$. Moreover, for every $\eta$ as above,
\begin{align*}
\int_{E_{0,T}}|u|\eta\,dm_1=(\hat{G}^{0,T}\nu,\eta)_{L^2} =
\langle \nu, \hat{G}^{0, T}\eta\rangle & \le
\|\nu\|_{TV}\|\hat{G}^{0,T}\eta\|_{\infty}\\
& \le
(\|\varphi\|_{L^1}+\|\mu\|_{TV})\|\hat{G}^{0,T}\eta\|_{\infty},
\end{align*}
so $u\in L^{1}(E_{0,T};\eta\cdot m_1)$. Letting
$n\rightarrow\infty$  in (\ref{eq1.14}) we get (\ref{eq3.22}).
\end{dow}

\begin{wn}
\label{wn.dfk} Let assumptions of Proposition \ref{stw1.6} hold
and let $u$  be a solution of \mbox{\rm{(\ref{eq1.10})}} in the
sense of duality. Then there exists an $m_1$-version of $u$
satisfying \mbox{\rm{(\ref{eq1.0})}}.
\end{wn}

\begin{uw}
Let $\{B^{(t)};t\in\BR\}$ be a family of non-negative
quasi-regular Dirichlet  forms  satisfying (\ref{eqp.0}). A
careful inspection of the proof of \cite[Theorem VI.1.2]{MR}
reveals that there exist a $B^{(0)}$-nest $\{E_k\}_{k\ge 1}$
consisting of compact metrizable sets in $E$ and a locally compact
separable metric space $Y^{\#}$ such that $Y^{\#}$ is a local
compactification of $Y=\bigcup_{k\ge 1}E_k$. Moreover, the trace
topologies of $E_k$ induced by $E$ and $Y^{\#}$ coincide and
$(B^{(s),\#},D(B^{(s),\#}))$, which is the image of
$(B^{(s)},D(B^{(s)}))$ under the inclusion map $i:Y\rightarrow
Y^{\#}$, is a regular Dirichlet form on $L^2(Y^{\#};m^{\#})$,
where $m^{\#}=m\circ i^{-1}$. By \cite[Theorem VI.1.6]{MR} the
Hunt process $\mathbb{M}^{(s),\#}=(\{P_{s,x}^{\#},\,x\in
E^{\#}\},\{X^{\#}_{s+t},\,t\ge 0\},\zeta^{\#})$ associated with
the regular form  $B^{(s),\#}$ is the trivial extension of the
special standard process $\mathbb{M}^{(s)}=(\{P_{s,x},\,x\in
E\},\{X_{s+t},\,t\ge 0\},\zeta)$ associated  with the form
$B^{(s)}$ Let  $\EE^{\#}$ be the time-dependent Dirichlet form on
$L^2(E^{1\#};m^{\#}_1)$, where $E^{1\#}=\BR\times E^{\#}$,
constructed from the family $\{B^{(s),\#};\, s\in\BR\}$ as in
Section 2 (see (\ref{eq2.02})). Then the process
$\mathbb{M}^{\#}=(\{\bX_{t}^{\#}, t\ge 0\},\{P_{z}^{\#},
z\in\BR\times E^{\#}\},\zeta^{\#})$ on
$\Omega'=\Omega\cup(E^{1\#}\setminus E^{1})$ associated with the
form $\EE^{\#}$ is given by
\[
\bX_{t}^{\#}(\omega)=\bX_{t}(\omega),\,\,t\ge 0,\,
\omega\in\Omega, \quad \bX_{t}^{\#}(\omega)=\omega, \,\, t\ge 0,\,
\omega\in E^{1\#}\setminus E^{1}
\]
and $P_{z}^{\#}=P_{z}$ for $z\in E^{1}$,
$P_{z}^{\#}=\delta_{\{z\}}$  for $z\in E^{1\#}\setminus E^{1}$. It
is clear that the trace topologies on $\BR\times E_{k}$ induced by
$\BR\times E$ and by $\BR\times Y^{\#}$ coincide. It follows that
\cite[Corollary VI.1.4]{MR} holds true for the form $\EE^{\#}$ and
capacity $\mbox{Cap}_{h,g}$ considered in \cite{MR} replaced by
$\mbox{Cap}_{\psi}$.
\end{uw}

\begin{uw}
\label{uw1.2} The above remark shows that one can apply the
so-called ``transfer method" (see \cite[Section VI]{MR},
\cite{KR:SM})   to the form $\EE$ defined by (\ref{eq2.02}).
Therefore the results of the present paper hold true for $\EE$
with $B^{(t)}$ being quasi-regular Dirichlet forms.
\end{uw}

\nsubsection{Semilinear equations with measure data}
\label{sec5}

In this section we assume that $\mu\in\mathcal{R}(E_{0, T})$,
$\delta_{\{T\}}\otimes\varphi\cdot m\in\mathcal{R}(E_{0, T})$ and
$f\in \BB(E_{0,T})$. In what follows given $u\in \BB(E_{0,T})$ we
set
\[
f_{u}(t,x)=f(t,x,u(t,x)), \quad (t,x)\in E_{0,T}.
\]

\subsubsection{General semi-Dirichlet forms}

\begin{df}
We say that $u$ is a solution of the Cauchy problem
\begin{equation}
\label{eq2.1} -\frac{\partial u}{\partial t}-L_{t}u =f(t, x, u)+
\mu, \quad u(T)=\varphi
\end{equation}
if $f_{u}\in\mathcal{R}(E_{0, T})$ and for q.e. $z\in E_{0,T}$\,,
\begin{equation}
\label{eq2.9}
u(z)=E_{z}\Big(\mathbf{1}_{\{\zeta>T-\tau(0)\}}\varphi(\bX_{T-\tau(0)})
+\int_{0}^{\zeta_{\tau}}f(\bX_{t}, u(\bX_{t}))\,dt
+\int_{0}^{\zeta_{\tau}}dA_{t}^{\mu}\Big).
\end{equation}
\end{df}

\begin{df}
We say that $f:E_{0,T}\rightarrow\BR$ is quasi-integrable  ($f\in
qL^{1}(E_{0, T};m_1)$ in notation) if $f\in \BB(E_{0,T})$ and
$P_{z}(\int_{0}^{\zeta_{\tau}}|f|(\bX_{r})\,dr<\infty)=1$ for q.e.
$z\in E_{0,T}$\,.
\end{df}
Let us consider the following hypotheses.
\begin{enumerate}
\item[(H1)]$u\mapsto f(t,x,u)$ is
continuous for every $(t, x)\in E_{0, T}$,
\item[(H2)]There is $\alpha\in\BR$ such that
\[
(f(t,x,y)-f(t,x,y'))(y-y')\le\alpha|y-y'|^{2}
\]
for every $(t,x)\in E_{0,T}$ and $y, y'\in\BR$.
\item[(H3)]$f(\cdot, 0)\in\mathcal{R}(E_{0, T})$.
\item[(H4)]$f(\cdot, y)\in qL^{1}(E_{0, T};m_1)$ for every $y\in\BR$.
\end{enumerate}

\begin{uw}
It is clear that $\mathcal{B}(E_{0,T})\cap\mathcal{R}(E_{0,T})
\subset qL^1(E_{0,T};m_1)$. By Proposition \ref{stw1.5}, under the
dual condition $(\Delta)$,
$L^1(E_{0,T};m_1)\subset\mathcal{R}(E_{0,T})$. It follows that
\begin{equation}
\label{eq4.3} L^1(E_{0,T};m_1)\subset qL^1(E_{0,T};m_1)
\end{equation}
under $(\Delta)$. Let us consider the following condition
\begin{equation}
\label{eq.qq} \forall\, \varepsilon>0\,\, \exists\,
F_{\varepsilon}\subset E_{0,T}, \,\,
F_{\varepsilon}\mbox{-closed},\,\,
\mbox{Cap}_{\psi}(E_{0,T}\setminus
F_{\varepsilon})<\varepsilon,\,\, \mathbf{1}_{F_{\varepsilon}}f\in
L^1(E_{0,T};m_1).
\end{equation}
Assume that $f\in\mathcal{B}(E_{0,T})$ satisfies (\ref{eq.qq}) and
the dual condition $(\Delta)$ holds. Let $\{F_{n}\}$ be an
increasing sequence of closed subsets of $E_{0,T}$ such that
Cap$_{\psi}(G_{n})\rightarrow 0$, where $G_{n}=E_{0,T}\setminus
F_{n}$. By (\ref{eq4.3}),
\begin{align*}
P_{m_1}(\int_{0}^{\zeta_{\tau}}|f|(\mathbf{X}_{r})\,dr=\infty) &
\le P_{m_1}(\int_{0}^{\zeta_{\tau}}
|f|\mathbf{1}_{F_{n}}(\mathbf{X}_{r})\,dr=\infty)\\
&\quad+P_{m_1}(\int_{0}^{\zeta_{\tau}}
|f|\mathbf{1}_{G_{n}}(\mathbf{X}_{r})\,dr=\infty).
\end{align*}
The first term on the right-hand side of the above inequality
equals zero. From the above and \cite[Remark IV.3.6]{Stannat} it
follows that
\begin{align*}
P_{m_1}(\int_{0}^{\zeta_{\tau}}|f|(\mathbf{X}_{r})\,dr=\infty)
=P_{m_1}(\int_{\sigma_{G_{n}}}^{\zeta_{\tau}}
|f|\mathbf{1}_{G_{n}}(\mathbf{X}_{r})\,dr=\infty)& \le
P_{m_1}(\lim_{n\rightarrow\infty}\sigma_{G_{n}}<\zeta_{\tau})\\
&=0.
\end{align*}
Consequently,
$P_{z}(\int_{0}^{\zeta_{\tau}}|f|(\mathbf{X}_{r})\,dr=\infty)=0$
for $m_1$-a.e., and hence for q.e. $z\in E_{0,T}$ by standard
argument. Thus $f\in qL^1(E_{0,T};m_1)$.
\end{uw}

From Proposition \ref{stw1.5} we know that if $\EE$ satisfies
$(\Delta)$ then $\MM_{0,b}(E_{0,T})\subset \mathcal{R}(E_{0,T})$.
The following example shows that the inclusion may be is strict.

\begin{prz}
Let $f$ be a non-negative measurable function on $E_{0,T}$. Then
\[
(G^{0,T}{f},1)_{L^2}=(f,\hat{G}^{0,T}1)_{L^2}.
\]
Let $D$ and $L_t$ be as in Example \ref{prz3.8}. Then for $x\in
D$,
\[
\hat{R}^{0,T}(s,x)\le c_1\int_{0}^{s}\!\int_{D}
t^{-d/2}\exp(\frac{-c_2|y-x|^2}{2t})\,dt\,dy \le
c_3\int_{D}G^1_{D}(x,y)\,dy\le c_{4}\delta(x),
\]
where $\delta(x)=\mbox{dist}(x,\partial D)$ and
$G^1_{D}(\cdot,\cdot)$ is the Green function for the operator
$\Delta$ on $D$. If $L_{t}=\Delta^{\alpha/2}$, where $\alpha\in
(0,2)$, then for $x\in D$ we have
\[
\hat{R}^{0,T}1(s,x)\le \int_{D}G^2_{D}(x,y)\,dy \le
c\delta^{\alpha/2}(x),
\]
where $G^2_{D}(\cdot,\cdot)$ is the Green function for the
operator $\Delta^{\alpha/2}$ on $D$ (For the last inequality see
\cite[Proposition 4.9]{Kulczycki}). We see that
$L^1(E_{0,T};\delta\cdot m_1)\subset\mathcal{R}(E_{0,T})$ if $L_t$
is the operator of Example \ref{prz3.8} and
$L^1(E_{0,T},\delta^{\alpha/2}\cdot m_1)\subset
\mathcal{R}(E_{0,T})$ if $L_{t}=\Delta^{\alpha/2}$, so in both
cases $\MM_{0,b}(E_{0,T})\subsetneq \mathcal{R}(E_{0,T})$.
\end{prz}

The next example shows that in general quasi-integrable functions
need not be locally integrable.

\begin{prz}
Let $L_t$ be as in Remark \ref{prz3.8} and let $D=\{x\in\BRD;\,
|x|<1\}$, $d\ge 2$.  Set $f(t,x)=|x|^{-d},\, (t,x)\in E_{0,T}$.
Direct calculation shows that
$\int_{0}^{T}\int_{B(0,\varepsilon)}f(t,x)\,dt\,dx=\infty$ for
every $\varepsilon\in (0,1)$, i.e. $f$ is not locally integrable.
It is, however, quasi-integrable, because if  $F_{n}=\{(t,x)\in
E_{0,T}; t\in [0,T],|x|\ge\frac1n\}$ then $\mathbf{1}_{F_{n}}f\in
L^{1}(E_{0,T};m_1)$, $n\ge1$. Moreover, if we set
$G_{n}=E_{0,T}\setminus F_{n}$ then $\bar{G}_{n+1}\subset G_{n}$
and $\bigcap_{n}G_{n}=\bigcap_{n}\bar{G}_{n}$. Since Cap$_{\psi}$
is a Choquet capacity (see \cite[Proposition III.2.8]{Stannat}),
it follows that
\[
\lim_{n\rightarrow \infty}\mbox{Cap}_{\psi}(G_{n})
=\mbox{Cap}_{\psi}(\bigcap_{n}G_{n})=\mbox{Cap}_{\psi}((0,T]\times\{0\})=0.
\]
\end{prz}

From the above example it follows in particular that in general
$\mathcal{R}(E_{0,T})\subsetneq qL^1(E_{0,T})$. For instance, the
inclusion is strict if $L$ is the Laplace operator on smooth
bounded domain, because in this case each function from
$\mathcal{R}(E_{0,T})$ is locally integrable thanks to the
positivity and continuity of the corresponding Green function.

Let $(\Omega,\FF=\{\FF_{t},t\in[0,T]\}, P)$ be a fixed stochastic
basis.   Suppose we are given an $\FF_T$ measurable random
variable $\xi$, an $\FF$ progressively measurable function
$F:\Omega\times [0, T]\times\BR\rightarrow\BR$ and an $\FF$
adapted c\`adl\`ag process $A$ of finite variation.

\begin{df}
We say that a pair $(Y,M)$ of processes on $[0,T]$ is a solution
of the backward stochastic differential equation
\begin{equation}
\label{eq2.2} Y_{t}=\xi+\int_{t}^{T}F(r, Y_{r})\,dr+
\int_{t}^{T}dA_{r}-\int_{t}^{T}dM_{r}, \quad t\in[0,T]
\end{equation}
(BSDE$(\xi,F+dA)$ in notation) if $Y$ is a progressively
measurable  process of class $\mbox{(D)}$, $t\rightarrow
F(t,Y_{t})\in L^{1}(0, T)$, $M$ is a $\FF$-martingale such that
$M_0=0$ and (\ref{eq2.2}) holds $P$-a.s.
\end{df}
We will need the following assumptions.
\begin{enumerate}
\item[(A1)]$y\rightarrow F(t, y)$ is continuous
for a.e. $t\in[0, T]$,
\item[(A2)]there is $\alpha\in\BR$ such that
\[
(F(t,y)-F(t,y'))(y-y')\le\alpha|y-y'|^{2}
\]
for a.e. $t\in[0,T]$ and every $y,y'\in\BR$,
\item[(A3)] $E\int_{0}^{T}|F(t,0)|\,dt<\infty$,
$E|\xi|+E|A|_{T}<\infty$ ($|A|_T$ denotes the variation of $A$ on
$[0,T]$),
\item[(A4)]$[0,T]\ni t\rightarrow F(t, y)\in L^{1}(0,T)$
for every $y \in\BR$.
\end{enumerate}

\begin{tw}
\label{tw2.1} Assume $\mbox{\rm{(A1)--(A4)}}$. Then there exists a
unique solution $(Y, M)$ of $\mbox{\rm{BSDE}}$$(\xi, f+dA)$.
Moreover, $Y, M\in S^{q}$, $q\in(0,1)$ and
\[
E\int_{0}^{T}|F(t, Y_{t})|\,dt \le C(\alpha,
T)\Big(E\int_{0}^{T}|F(t,0)|\,dt+ E\int_{0}^{T}d|A|_{t}\Big).
\]
\end{tw}
\begin{proof}
By using the standard change of variable one can reduce the proof
to the case where $\alpha=0$ in (H2). But then the desired result
follows from \cite[Theorem 2.7]{KR:JFA}.
\end{proof}
\begin{df}
Let $z\in E$. We say that a pair $(Y, M)$ is a solution of
BSDE$_{z}(\varphi,f+d\mu)$ if $(Y, M)$ is a solution of the BSDE
\[
Y_{t}=\mathbf{1}_{\{\zeta>T-\tau(0)\}}\varphi(\bX_{\zeta_{\tau}})
+\int_{t}^{\zeta_{\tau}}f(\bX_r, Y_{r})\,dr+
\int_{t}^{\zeta_{\tau}}dA^{\mu}_{r}-\int_{t}^{\zeta_{\tau}}dM_{r},
\quad t\in[0,\zeta_{\tau}]
\]
on the probability space $(\Omega,\FF,P_{z})$.
\end{df}

\begin{stw}
\label{stw2.1} Assume $\mbox{\rm{(H1)--(H4)}}$. Then for q.e.
$z\in E_{0, T}$ there exists a unique solution $(Y^{z},M^{z})$  of
$\mbox{\rm{BSDE}}_{z}(\varphi, f+d\mu)$. Moreover, there exists a
pair of processes $(Y, M)$ such that for q.e. $z\in E_{0,T}$,
\[
(Y_{t}, M_{t})=(Y_{t}^{z}, M^{z}_{t}), \quad t\in[0, T-\tau(0)],
\quad P_{z}\mbox{\rm-a.s.}
\]
\end{stw}
\begin{proof}
If $\varphi,f,\mu$ satisfy (H1)--(H4) then
$\xi=\mathbf{1}_{\{\zeta>T-\tau(0)\}}\varphi(\bX_{\zeta_{\tau}})$,
$F=f(\cdot,\bX,\cdot)$, $A=A^{\mu}$ satisfy (A1)--(A4) under the
measure $P_{z}$ for q.e. $z\in E_{0,T}$. Therefore the first part
of the proposition follows from Theorem \ref{tw2.1}. The second
part follows from \cite[Remark 3.6]{KR:JFA}.
\end{proof}

\begin{lm}
\label{lm2.1} Assume \mbox{\rm(H1)--(H4)} and let $(Y, M)$ be the
pair of Proposition \ref{stw2.1}. Then for q.e. $z\in E_{0, T}$
and every $h\in[0,T-\tau(0)]$,
\[
Y_{t}\circ \theta_{h}=Y_{t+h}, \quad t\in[0, T-\tau(0)-h],
\quad P_{z}\mbox{\rm{-a.s.}}
\]
\end{lm}
\begin{proof}
Since for q.e. $z\in E_{0,T}$ the solution of
BSDE$_{z}(\varphi,f+d\mu)$ is unique, to prove the proposition it
suffices to repeat the proof of \cite[Proposition 3.5]{KR:JFA}
(see also the proof of \cite[Proposition 3.24]{KR:JEE}).
\end{proof}

\begin{tw}
\label{tw2.2} Assume \mbox{\rm(H1)--(H4)}. Then there exists a
unique solution $u$ of $\mbox{\rm{(\ref{eq2.1})}}$. Moreover,  for
q.e. $z\in E_{0, T}$ there exists a unique solution $(Y^{z},
M^{z})$ of $\mbox{\rm{BSDE}}_{z}(\varphi,f+d\mu)$. In fact,
\[
Y_{t}^{z}=u(\bX_{t}), \quad t\in[0, \zeta_{\tau}],
\]
\[
M_{t}^{z}=E_{z}\Big(\mathbf{1}_{\{\zeta>T-\tau(0)\}}\varphi(\bX_{T-\tau(0)})
+\int_{0}^{\zeta_{\tau}}f_{u}(\bX_{r})\,dr
+\int_{0}^{\zeta_{\tau}}dA_{r}^{\mu}|\FF_{t}\Big)-u(\bX_{0}).
\]
\end{tw}
\begin{dow}
By Proposition \ref{stw2.1} for q.e. $z\in E_{0, T}$ there exists
a unique solution $(Y^{z}, M^{z})$ of BSDE$_{z}(\varphi,f+d\mu)$.
Let $(Y,M)$ be the pair of Proposition \ref{stw2.1} and let
$u(z)=E_{z}Y_{0}$. Then by Lemma \ref{lm2.1} and the strong Markov
property,
\[
u(\bX_{t})=E_{\bX_{t}}Y_{0}= E_{z}(Y_{0}\circ\theta_{t}|\FF_{t})
=E_{z}(Y_{t}|\FF_{t})=Y_{t}
\]
for every $t\in[0,T-\tau(0)]$. From this, (H3), Theorem
\ref{tw2.2} and the definition of a solution of
BSDE$_{z}(\varphi,f+d\mu)$ we deduce that (\ref{eq2.9}) is
satisfied for q.e. $z\in E_{0,T}$, i.e. $u$ is a solution of
(\ref{eq2.1}). By Proposition \ref{stw1.1}, $u$ is
quasi-c\`adl\`ag. Therefore $Y_{t}=u(\bX_{t})$, $t\in[0,
\zeta_{\tau}]$. From this and the definition of a solution of
BSDE$_{z}(\varphi,f+d\mu)$ the representation formula for $M^{z}$
immediately follows. Suppose now that $v$ is another solution of
(\ref{eq2.1}). Then by the strong Markov property,
\[
v(\bX_{t})=\mathbf{1}_{\{\zeta>T-\tau(0)\}}\varphi(\bX_{T-\tau(0)})
+\int_{t}^{\zeta_{\tau}}f_{v}(\bX_{r})\,dr
+\int_{t}^{\zeta_{\tau}}dA_{r}^{\mu}-\int_{t}^{\zeta_{\tau}}d\bar{M}_{r},
\]
where $\bar{M}$ is a c\`adl\`ag and independent of $z$ version of
the martingale $N^{z}$ given by
\[
N_{t}^{z}=E_{z}\Big(\mathbf{1}_{\{\zeta>T-\tau(0)\}}\varphi(\bX_{T-\tau(0)})
+\int_{0}^{\zeta_{\tau}}f_{v}(\bX_{r})\,dr+
\int_{0}^{\zeta_{\tau}}dA_{r}^{\mu}|\FF_{t}\Big)- v(\bX_{0})
\]
(Existence of such version follows from \cite[Lemma
A.3.5]{Fukushima}). We see that the pair $(v(\bX),\bar{M})$ is a
solution of BSDE$_{z}(\varphi, f+d\mu)$ for q.e. $z\in E_{0, T}$.
Consequently,  $u=v$ q.e. by Proposition \ref{stw2.1}.
\end{dow}

\begin{wn}
\label{wn4.444} Let assumptions of Theorem \ref{tw2.2} hold and
let  $u_{i}$ be  a solution of \mbox{\rm(\ref{eq2.1})} with
terminal condition $\varphi_{i}$, and right-hand side
$f_{i}+d\mu_{i}$, $i=1,2$. If $\varphi_{1}\le\varphi_{2}$
$m_{1}$-a.e., $\mu_{1}\le\mu_{2}$ and either $f_{1}$ satisfies
$\mbox{\rm{(H2)}}$ and $f_{1,u_{2}}\le f_{2,u_{2}}$ $m_{1}$-a.e.
or $f_{2}$ satisfies $\mbox{\rm{(H2)}}$ and $f_{1, u_{1}}\le f_{2,
u_{1}}$ $m_{1}$-a.e., then then $ u_{1}(z)\le u_{2}(z)$ for q.e.
$z\in E_{0, T}$.
\end{wn}
\begin{dow}
Follows immediately from Theorem \ref{tw2.2} and
\cite[Proposition 2.1]{KR:JFA}.
\end{dow}

\begin{stw}
\label{stw4.6} Let assumptions of Theorem \ref{tw2.2} hold,
$\varphi\in L^{1}(E;m)$, $\mu\in\MM_{0,b}(E_{0,T})$ and for some
$\gamma\ge0$ the form $\EE_{\gamma}$ has the dual Markov property.
Then if $u$ is a solution of \mbox{\rm(\ref{eq2.1})} then
$f_{u}\in L^{1}(E_{0,T};m_{1})$ and
\[
\|f_{u}\|_{L^{1}}\le C(\alpha, T, \gamma)
(\|\mu\|+\|\varphi\|_{L^{1}}+\|f(\cdot, 0)\|_{L^{1}}).
\]
\end{stw}
\begin{dow}
Let $(Y,M)$ be as in Proposition \ref{stw2.1} and let
$(\tilde{Y}_{t},\tilde{M}_{t})=(e^{\gamma t}Y_{t},e^{\gamma
t}M_{t})$, $t\in [0,\zeta_{\tau}]$. Applying It\^o's formula shows
that $(\tilde{Y},\tilde{M})$ is a solution of
BSDE$_{z}(\tilde{\varphi},\tilde{f}+d\tilde{\mu})$ with
$\tilde{\varphi}(x)=e^{\gamma \zeta_{\tau}}$,
$\tilde{f}(t,x,y)=e^{\gamma t}f(t,x,y)-\gamma y$ and
$d\tilde{\mu}(t,x)=e^{\gamma t}\,d\mu(t,x)$. By  Theorem
\ref{tw2.1} applied to the pair $(\tilde{Y},\tilde{M})$,
\begin{align*}
E_{z}\int_{0}^{\zeta_{\tau}}e^{\gamma
t}|f(\mathbf{X}_{t},Y_{t})|\,dt &\le
C(\alpha,T)\Big(E_{z}\mathbf{1}_{\{\zeta>T-\tau(0)\}}
e^{\gamma\zeta_{\tau}}\varphi(\mathbf{X}_{T-\tau(0)})\\
&\quad+E_{z}\int_{0}^{\zeta_{\tau}}e^{\gamma
t}|f(\mathbf{X}_{t},0)|\,dt +\gamma
E_{z}\int_{0}^{\zeta_{\tau}}e^{\gamma t}|Y_{t}|\,dt\Big)
\end{align*}
for q.e. $z\in E_{0,T}$. Since $Y_{t}=u(\mathbf{X}_{t})$, $t\in
[0,\zeta_{\tau}]$, $P_{z}$-a.s. for q.e. $z\in E_{0,T}$ by Theorem
\ref{tw2.2} and $\EE_{\gamma}$ has the dual Markov property, it
follows from the above inequality and Proposition \ref{stw1.4}
that
\[
\|f_{u}\|_{L^1}\le C(\alpha,T)(\|\varphi\|_{L^1}
+\|f(\cdot,0)\|_{L^1}+\gamma\|u\|_{L^1}).
\]
From this and  Theorem \ref{tw3.8} we get the desired inequality.
\end{dow}

\begin{wn}
\label{wn4.7} Let assumptions of Proposition \ref{stw4.6} hold. If
$u$ is a solution  of \mbox{\rm(\ref{eq2.1})} then $u\in
L^{1}(E_{0,T},m_1)$ and $T_{k}(u)\in\FF_{0,T}$ for $k\ge0$.
Moreover, \mbox{\rm(\ref{eq3.15})} and \mbox{\rm(\ref{eq3.16})}
hold true with $\mu$ replaced by $\mu+f(\cdot,0)\cdot m$.
\end{wn}
\begin{dow}
Follows from Theorem \ref{tw3.8} and Proposition \ref{stw4.6}.
\end{dow}

\subsubsection{Semi-Dirichlet forms satisfying the duality condition}

Let us recall that in Section \ref{sec4} we have defined a
solution in the sense of duality of linear equations. In the
semilinear case we adopt the following natural definition.

\begin{df}
Let $\varphi\in L^{1}(E; m)$, $\mu\in\mathcal{M}_{0, b}(E_{0, T})$
and assume that $\EE$ satisfies the dual condition $(\Delta)$. We
say that a measurable function $u:E_{0,T}\rightarrow\BR$ is a
solution of (\ref{eq2.1}) in the sense of duality if $f_{u}\in
L^{1}(E_{0,T};m_{1})$ and (\ref{eq3.22}) is satisfied with $\mu$
replaced by $f_{u}\cdot m_1+\mu$.
\end{df}

\begin{tw}
Assume \mbox{\rm{(H1)--(H4)}} and that there is $\gamma\ge 0$ such
that $\EE_{\gamma}$ has the dual Markov property. Then there
exists a unique solution of \mbox{\rm(\ref{eq2.1})} in the sense
of duality.
\end{tw}
\begin{dow}
The existence part follows from Theorem \ref{tw2.2} and
Corollaries \ref{wn.dfk} and \ref{wn4.7}. The uniqueness follows
from Corollaries \ref{wn.dfk} and \ref{wn4.444}.
\end{dow}

\begin{prz}
Let $\alpha$ be a measurable function on $\BRD$ such that
$\alpha_1\le\alpha(x)\le\alpha_2$, $x\in\BRD$, for some constants
$0<\alpha_{1}\le\alpha_{2}<2$.  Let $L_{t}=L=\Delta^{\alpha(x)}$,
i.e. $L$ is a pseudodifferential operator such that
\[
Lu(x)=\int_{\BR^d}
e^{ix\xi}|\xi|^{\alpha(x)}\hat{u}(\xi)\,d\xi,\quad u\in
C^{\infty}_c(\BR^d).
\]
For $r>0$ set $\beta(r)=\sup_{|x-y|\le r}|\alpha(x)-\alpha(y)|$.
By \cite[Proposition 3.1]{SW}, if
\[
\int_0^1\frac{(\beta(r)|\log r|)^2}{r^{1+\alpha_2}}\,dr<\infty
\]
then the form  $B^{(t)}=B$ associated with $L$ is a regular
semi-Dirichlet form. It is known (see \cite{FU,SW}) that for
$u,v\in C_{c}^{\infty}(\BRD)$ the form $B$ is given by
\[
B(u,v)=-\int_{\BRD}\int_{\{z\neq 0\}}w(x)v(x)(u(x+z)-u(x)-\nabla
u(x)\cdot z\mathbf{1}_{\{|z|\le 1\}}(z))
|z|^{-d-\alpha(x)}\,dx\,dz,
\]
where
\[
w(x)=\alpha(x)2^{\alpha(x)-1}\frac{\Gamma(\frac12\alpha(x)
+\frac12 d)}{\pi^{d/2}\Gamma(1-\frac12\alpha(x))}.
\]
By \cite[Theorem 2.1]{SW},
\begin{equation}
\label{eq5.6} L^{*}u(x)=\Lambda u(x)+\kappa(x)u(x),\quad u\in
C_{c}^{\infty}(\BRD)
\end{equation}
for some measurable function $\kappa$ and some operator $\Lambda$
associated with a semi-Dirichlet form. By \cite[Remark 3.2]{SW},
under the additional condition that $\alpha\in C_{b}^2(\BRD)$ the
function $\kappa$ is bounded on $\BRD$. Therefore from
(\ref{eq5.6}) it follows that there exists $\gamma\ge 0$ such that
$B_{\gamma}$ has the dual Markov property, which implies that
$\EE_{\gamma}$ has the dual Markov property.
\end{prz}
{\bf Acknowledgements}
\medskip\\
Research supported by Polish NCN grant no. 2012/07/D/ST1/02107.

\end{document}